\documentclass{amsart}

\usepackage{xr}
\usepackage[parfill]{parskip}
\usepackage{graphicx} 
\usepackage{color} 
\usepackage{hyperref} 
\usepackage{graphicx}  
\usepackage{verbatim}  
\hypersetup{pdfpagemode=FullScreen,colorlinks=true}
\usepackage{appendix, url}
\usepackage{amsmath, amssymb, mathrsfs, amsfonts, amsthm} 
\usepackage{tikz-cd}
\DeclareMathOperator{\lie}{lie}
\DeclareMathOperator{\qobs}{q-obs}

\DeclareMathOperator{\Aut}{Aut}
\DeclareMathOperator{\Vol}{Vol}

\DeclareMathOperator{\energy}{energy}

\DeclareMathOperator{\hocolim}{hocolim}
\DeclareMathOperator{\image}{image}
\DeclareMathOperator{\D}{\Delta}

\newtheorem {theorem} {Theorem} [section]

\newtheorem{lemma}[theorem] {Lemma}

\newtheorem {definition} [theorem] {Definition}

\newtheorem {corollary}[theorem]  {Corollary} 
\newtheorem {example} [theorem]  {Example}
  
\newtheorem {notation}[theorem] {Notation}
\newtheorem {terminology}[theorem]{Terminology}
\newtheorem {remark} [theorem] {Remark}
\numberwithin {equation} {section}

\DeclareMathOperator{\obj}{obj}
\DeclareMathOperator{\colim}{colim} 
\begin{document} 
\href{http://yashamon.github.io/web2/papers/fukayaII.pdf}{Direct link to author's version}
\author {Yasha Savelyev} 
\today
\address{University of Colima, CUICBAS, Bernal Díaz del
Castillo 340,
Col. Villas San Sebastian,
28045, Colima, Colima,
Mexico}
\email{yasha.savelyev@gmail.com} 
\title [Global Fukaya Category II]{Global Fukaya
category II:  applications}  
\begin{abstract} 
To paraphrase, part I constructs a bundle of $A _{\infty}$
categories given the input of a Hamiltonian fibration over
a smooth manifold. Here we show that this bundle is
generally non-trivial by a sample
computation. One principal
application is differential geometric, and the other is
about algebraic $K$-theory of the integers and the rationals.  We find  new curvature
constraint phenomena for smooth and singular
$\mathcal{G}$-connections on principal $\mathcal{G}$-bundles over $S
^{4}$, where  $\mathcal{G}$ is  $\operatorname {PU} (2)$ or
$\operatorname {Ham} (S ^{2} )$.  Even for the classical
group $\operatorname {PU} (2)$ these phenomena are
inaccessible to known techniques like the Yang-Mills
theory.  
The above mentioned computation is the geometric component used to show that
the categorified algebraic $K$-theory of
the integers and the rationals, defined in ~\cite{cite_SavelyevAlgKtheory} 
following To\"en, admits a $\mathbb{Z} $ injection in degree
$4$. This gives a path from Floer theory to
number theory.
\end{abstract}
\maketitle
\section{Introduction}
Given a smooth Hamiltonian fibration $M \hookrightarrow
P \to X$, with $M$ monotone, and $X$ any smooth manifold, 
in Part I, that is in
\cite{cite_SavelyevGlobalFukayaCategoryI},
we have constructed a functor which may be understood as
giving a fibration of $A _{\infty }$ categories over $X$, with
fiber the Fukaya category of $M$. This construction is
to the Fukaya category as vector bundles are to vector
spaces. It will not be necessary to first read Part I to
follow this paper, as we will outline most concepts. 

More specifically, denote by $\Delta^{} (X)$ the smooth simplex
category of $X$,  
whose objects are smooth maps $\Sigma: \Delta^{d} \to X $.
See Section \ref{sec_preliminaries} for more details. Denote
by $w{A} _{\infty}Cat ^{\mathbb{Z} _{2}} _{k}$ the category
of strictly unital $A _{\infty}$ categories over
a commutative ring $k$, with $\mathbb{Z} _{2}$ graded hom
complexes, with morphisms quasi-isomorphisms. 

Then on the universal level, in Part I we constructed a functor:
\begin{equation} \label{equation_F}
F _{k}: \Delta^{} (B \operatorname {Ham} (M,  \omega)) \to
w{A} _{\infty}Cat ^{\mathbb{Z}
_{2}} _{k}, 
\end{equation}
having some additional structure, which facilitates some 
``homotopy theory''.
Here we show, for $k=\mathbb{Z}, \mathbb{Q} $, that this functor is
(generally) ``homotopically'' non-trivial, and give some applications
for this fact. Homotopic non-triviality is more specifically
a condition on the concordance type of $F _{k}$, see Definition
\ref{def_concordancefunctor}.


As explained in Part I, the above homotopy non-triviality may be understood to say that there is a non-trivial 
``homotopy coherent'' action of $\operatorname {Ham}
(M,  \omega) $ on $\operatorname {Fuk} (M,  \omega) $.  
Existence of such an action in a somewhat weaker form (on
the level $E _{2}$ algebras rather than space level) has
been independently conjectured by Teleman, ICM 2014,
post-factum of appearance of Part I. 

\subsection{Method of Cartesian and Kan fibrations} \label{sec_Method of Cartesian fibration}
Given our Hamiltonian bundle $M \hookrightarrow P \to X$, 
analogously to \eqref{equation_F}, we get a functor:
\begin{equation*}
F _{P, k}: \Delta^{} (X) \to   w{A} _{\infty}Cat ^{\mathbb{Z}
_{2}} _{k}.
\end{equation*}

Let $ \operatorname {NFuk} (M,
\omega)$, denote the Faonte-Lurie $A _{\infty}
$ nerve of $\operatorname {Fuk} (M, \omega),$ (see Appendix
A.4 of Part I.)
Denote by $X _{\bullet }$ the smooth singular set, that is
the simplicial set with $n$-simplices smooth maps
$\Delta^{n} \to X $. 
As we will recall in Section \ref{section.application} $F
_{P,\mathbb{Q}}$ induces a simplicial fibration:
\begin{equation*}
\operatorname {NFuk} (M,  \omega)
\hookrightarrow {\operatorname {Fuk} _{\infty} (P)}
\xrightarrow{p _{\bullet }} X _{\bullet}.
\end{equation*}
This simplicial fibration is furthermore a Cartesian
fibration (in this case equivalently categorical fibration).

We show that for $P$ a non-trivial Hamiltonian $S ^{2} $ fibration over $S ^{4} $, the maximal Kan sub-fibration of $\operatorname {Fuk} _{\infty} (P) $ is non-trivial.
In particular, $\operatorname {Fuk} _{\infty} (P) $ is
non-trivial as a Cartesian fibration. The latter
readily implies:  
\begin{theorem} \label{thm:main} For $M=S ^{2}$ the functors
$F _{\mathbb{Z}}, F _{\mathbb{Q} }$ are not null-concordant.
Furthermore, $F _{P,\mathbb{Q} }$ is
not concordant to $F _{P', \mathbb{Q} }$ if $P \not\simeq P'$
are Hamiltonian $S ^{2}$ bundles over $S ^{4}$.
\end{theorem}
The proof of the theorem above makes use of a certain computation
from ~\cite{cite_SavelyevQuantumMaslov}, using the notion of
quantum Maslov classes.  
The following is a partial corollary, the proof of which will
only be outlined here, due to some algebraic dependencies
that are beyond the scope here. 
\begin{theorem} \label{cor_ktheory}
There is a natural injection 
$$\mathbb{Z} \to K ^{Cat, \mathbb{Z} _{2}} _{4} (k),$$ where
$k=\mathbb{Z}, \mathbb{Q} $ and the right hand side is the categorified algebraic
$K$-theory group, introduced in ~\cite{cite_SavelyevAlgKtheory},
following ideas of To\"en
~\cite{cite_ToenThehomotopytheoryofdg-categoriesandderivedMoritatheory}.
\end{theorem}
The groups $K ^{Cat, \mathbb{Z} _{2}} _{n} (\mathbb{Z})$
encode some arithmetic information, (at the moment
mysterious) that should be a strict
refinement of the basic algebraic $K$-theory groups $K _{n}
(\mathbb{Z} )$. The latter groups are only partially computed, but at least we do have some idea of the arithmetic information
contained therein.
It is a theorem of Rognes ~\cite{cite_RognesK4Zistrivial}
that $K _{4} (\mathbb{Z}) =0 $.  As
explained in ~\cite[Example 1]{cite_SavelyevAlgKtheory}, this
together with the result above implies that the natural
trace maps from
$K ^{Cat, \mathbb{Z} _{2}}  (\mathbb{Z}
)$ to $K (\mathbb{Z} ) $ have a kernel. So we get
a connection from Floer theory to algebraic $K$-theory and 
arithmetic. 
\subsection{An application to curvature bounds of
singular and smooth connections}
\subsubsection {A non-metric measure of curvature} \label{sec:non-metric}
The geometry of our calculation will naturally tie in with
some theory of singular connections, which we now outline.
Let $G$ as above be a Frechet Lie group,
we denote by $\lie G$ its Lie algebra. Let
$$\mathfrak n: \lie G   \to \mathbb{R}$$ be an $Ad$ invariant Finsler norm.  For a principal $G$-bundle $P$ over a Riemann surface $(S,j)$, and given a $G$ connection $\mathcal{A}$ on $P$ define a 2-form $\alpha _{\mathcal{A}} $ on $S$ by: 
\begin{equation*}
   \alpha _{\mathcal{A}}  (v,jv) = \mathfrak n (R _{\mathcal{A}} (v, jv)),
\end{equation*}
where $R _{\mathcal{A}} $ is the curvature 2-form of
$\mathcal{A}$. More specifically, the latter form has the
properties: $$R _{\mathcal{A}} (v, w) \in \lie \Aut P_{z},
$$ for $z \in S$, $v,w \in T _{z} S  $, $P _{z}$ the fiber
of $P$ over $z$, $\Aut P _{z} \simeq G $ the group of
$G$-bundle automorphisms of $P _{z} $, and where $\simeq$ means non-canonical group isomorphism.

Define
\begin{equation} \label{eq:areasurface}
\energy _{\mathfrak n}  (\mathcal{A}) := \int _{S} \alpha _{\mathcal{A}}.
\end{equation}
In the case $\mathcal{A}$ is singular with singular set $C$,
$\alpha _{\mathcal{A}} $  is defined on $S -C$, so we 
define 
\begin{equation*}
\energy _{\mathfrak n} (\mathcal{A}) := \int _{S - C} \alpha _{\mathcal{A}},
\end{equation*}
with the right-hand side now being an extended integral.
This $\energy$ is a non-metric measurement meaning that no Riemannian metric on $S$ is needed.

It is possible to extend the functional above to
a functional on the space $\mathcal{C}$ of $G$-connections
on principal $G$ bundles $P \to \Delta ^{n} $. It may seem
that $\Delta ^{n} $ has no connection to Riemann surfaces,
but in fact there is an intriguing such connection.  
Let ${\mathcal {S}} _{d} $ denote the universal curve over
$\overline{\mathcal{R}} _{d} $ - the moduli space of complex
structures on the disk with $d+1$ punctures on the boundary.
And let ${\mathcal {S}} ^{\circ} _{d} $ denote ${\mathcal {S}} _{d} $, with nodal points of the fibers removed. 
Then it is shown in Part I that there are certain axiomatized systems of maps:
\begin{equation*}
{u}: {\mathcal {S}} ^{\circ}  _{d}  \to \Delta
^{n}, \quad \text{with $d,n$ varying}.  
\end{equation*}
Such a system is uniquely determined up to suitable homotopy, and is referred to as $\mathcal{U}$. 

There is then a natural functional:
\begin{align} \label{eq:areaG}
   \energy _{\mathcal{U}}: \mathcal{C} \to \mathbb{R} _{\geq 0},
\end{align}
defined with respect to a choice of $\mathcal{U}$, see
Section \ref{sec:qcurvature}.  When $n=2$ it is just the
$\energy$ functional as previously defined.

\subsubsection{Curvature bounds} \label{sec_Curvature bounds}
The following basic result can be formally understood as a
corollary of Theorem \ref{thm:lowerboundsingular}. But it is
more elementary to see it as a corollary of  
Theorem \ref{prop:alternative}.    
\begin{corollary} [Of Theorem \ref{thm:lowerboundsingular}
  and of Theorem \ref{prop:alternative}] 
   \label{corol:example}  Let $P$ be a non-trivial
   $\operatorname {PU} (2) $ (or $\operatorname {Ham} (S ^{2}) $)  bundle $P
   \to S ^{4}$, let $D ^{4}
   _{\pm}$ be as above and 
let $\mathcal{A} $ be a smooth $PU (2) $ or $\operatorname
{Ham} (S ^{2},  \omega)  $ connection 
on $P$.  For any $\mathcal{U} $: $$\energy  _{\mathcal{U}
   }(\mathcal{A} ) \,|\, _{D ^{4} 
   _{-}} < \frac{1}{2} \implies \energy  _{\mathcal{U} }(\mathcal{A} ) \,|\, _{D ^{4}
   _{+}} \geq \frac{1}{2}.$$
\end{corollary}  
\begin{remark} \label{rem_}
The proof of even the above corollary 
traverses the entirety of the theory here. As this
is a very elementary result we may hope for a
simpler argument. 
In the case of $PU (2) $, one idea might be to replace Floer theory, used here, by the technically
simpler mathematical Yang-Mills theory over surfaces
~\cite{cite_AtiyahBottTheYang-MillsequationsoverRiemannsurfaces}. If we want to mimic the argument presented in this paper, then we should first
extend Yang-Mills theory to work with $G$-bundles
over surfaces with corners and holonomy constraints over
boundary. This should be possible, but beyond this things
are unclear, since, as we also use certain abstract
algebraic topology to glue various analytic data, and it is not clear how this would work for Yang-Mills theory.  

\end{remark}


\subsubsection {Abstract resolutions of singular
connections} 
We can use the computation
of Theorem \ref{thm:main} to obtain lower bounds for
the curvature of certain types of singular connections.  
\begin{definition} \label{def:basicSingular}
Let ${G} \hookrightarrow P \to X$ be a principal $G$ bundle, where $G$ is a Frechet Lie group. A \textbf{\emph{singular $G$-connection}} on $P$ is a  closed subset $C \subset X$, and a smooth Ehresmann $G$-connection 
 $\mathcal{A}$ on $P| _{X-C} $.
\end{definition}
The above definition is basic, as one often puts additional
conditions, see for instance
\cite{cite_ReeseHarverSingular},
\cite{cite_SingularSibnerSibner}.

Avoiding generality, suppose that
$\mathcal{A}$ is a singular $G$-connection on a
principal $G$-bundle $P \to S ^{n}$, with a single
singularity $x _{0} $. We will show that it is
possible to control the curvature of the singular
connection $\mathcal{A}$ if we impose a certain
structure on the singularity of $\mathcal{A}$.
The simplest way to do this is to ask for
existence of a certain kind of abstract
resolution.  

First, a simplicial $G$-connection $\mathcal{A} $ on $P$, as defined in Section
\ref{section:simplicialconnections}, is 
a functorial assignment $$\Sigma \mapsto \mathcal{A}
_{\Sigma },$$ where $\mathcal{A} _{\Sigma }$ is a smooth
$G$-connection on $\Sigma ^{*} P$, for each smooth $$\Sigma: \Delta ^{n} \to S ^{n}. $$
\begin{definition} \label{def:resolution}
For $\mathcal{A}, P$ as above a
\textbf{\emph{simplicial resolution}} of $\mathcal{A} $ is a simplicial $G$ connection $\mathcal{A} ^{res} $ on $P$, with
the following property. Let $\Sigma _{0}: \Delta ^{n} \to
S ^{n}  $ represent the generator of $\pi _{n} (S ^{n},
x _{0}) $ (cf. Appendix \ref{appendix:Kan}), then  $$(\Sigma
 _{0}   | _{interior \, \Delta ^{n}}) ^{*}  \mathcal{A}
 = inc ^{*} \mathcal{A} ^{res} _{\Sigma _{0} }, $$  
where $inc: {interior \, \Delta ^{n}} \to \Delta^{n} $ is
the inclusion map.
\end{definition}

In the following theorem $G=PU (2), n=4$ and the
norm $\mathfrak n$ on $\lie PU (2)$ will be taken to be
the operator norm, normalized so that the Finsler
length of the shortest one parameter subgroup from
$id$ to $-id$ is $\frac{1}{2}$. We will omit $\mathfrak n$  in notation. We also impose an additional constraint on $\mathcal{A} ^{res} $, 
so that the curvature at ``$\infty$'' is bounded by a threshold, which means the following. Let $$\Sigma _{\infty}: \Delta ^{4} \to x _{0} $$ be the constant map. Suppose that:
\begin{equation*}
   \energy _{\mathcal{U}} \mathcal{A} ^{res} _{\Sigma _{\infty}} < 1/2,
\end{equation*}
and suppose for simplicity that $\mathcal{A} ^{res} _{\Sigma
_{\infty}}$ is trivial along the edges of $\Delta ^{4} $,
later on this condition is relaxed, see Proposition
\ref{prop:alternative}.  (This condition can be completely
removed but at the cost of significant additional
complexity.) 

%
We say in this case that $\mathcal{A} ^{res} $ is
a \textbf{\emph{sub-quantum resolution}}. The following is
proved in Section \ref{sec:qcurvature}. 
\begin{theorem}  \label{thm:lowerboundsingular}
Let $P \to S ^{4} $ be a non-trivial principal $PU (2)$ bundle.
Let $\mathcal{A}$ be a singular $PU (2)$-connection on $P$ with a single singularity at $x _{0} $. Then for any sub-quantum resolution $\mathcal{A} ^{res}  $ 
of $\mathcal{A}$  and for any $\mathcal{U}$ as above 
$$\energy _{\mathcal{U}}   (\mathcal{A} ^{res}  _{\Sigma _{0}} ) \geq 1/2. $$     
%
\end{theorem}
The theorem has certain extensions to Hamiltonian
singular connections $\mathcal{A}$, understanding
$P$ as a principal $\operatorname {Ham}  (S ^{2} )$ bundle, 
see Section \ref{sec:qcurvature}. 

\begin{example}
   \label{example:singularconnection} Let
   $P$ be as above, and $\mathcal{A}' $   be an
   ordinary
   smooth $PU (2) $ connection on  $P$. Express
   $S ^{4} $ as a union of sub-balls $D ^{4}
   _{\pm} \subset S ^{4}$,
    intersecting only in the boundary. Suppose that we have the 
   property that $\energy  _{\mathcal{U}}(\mathcal{A} ') \,|\, _{D ^{4} _{-}} < \frac{1}{2}.$  Let $\mathcal{A} $ be
   the singular connection on $P$ obtained as  the
   push-forward of $\mathcal{A} '$ by the bundle
   map $\widetilde{q}: P \to P$ over the singular smooth map
   $q: S ^{4} \to S ^{4}$ taking $D ^{4} _{-}$ to a
   single point $\infty \in S ^{4}$, with $q|
   _{interior \, D ^{4} _{-}}$ an
   immersion.  
   Then
   $\mathcal{A} $ has a sub-quantum
   resolution $\mathcal{A} ^{res} $ by
   construction.
   In this case, the theorem above
    simply yields that  $\energy _{\mathcal{U}}   (\mathcal{A}' _{D ^{4}_+}   ) \geq 1/2. $ 
\end{example} 
\begin{remark} \label{rem_}
There are possible physical interpretations for singular
connections, as appearing in the context here.   A $PU (2) $
connection $\mathcal{A} $ on $P$ in physical terms represents
a Yang-Mills field on the space-time $S ^{4}$.
When the space-time has a black hole singularity,
the fields solving the Einstein-Yang-Mills
equations (mathematically connections as above) likewise
develop singularities.  There is a wealth of
physics literature on this subject, here is one sample reference ~\cite{cite_YangMillsBlackHoles}. 
 As quantum gravity is often related to simplicial
ideas, it is not inconceivable that the
mathematical sub-quantum
resolution condition above also has a (quantum gravity
theoretic) physical interpretation.    
\end{remark}

At this point the reader may be curious why
Theorem \ref{thm:main} has something to do with
Theorem \ref{thm:lowerboundsingular}. We cannot give the full story,
but the idea is that the Cartesian fibration $Fuk
_{\infty} (P) $ only sees the principal bundle $P$
(and its curvature) by the behavior of certain
holomorphic curves. When one has the sub-quantum
condition on the curvature of $\mathcal{A} ^{res}
_{\Sigma _{\infty}}$, certain holomorphic curves
are ruled out so that from the view point of $Fuk
_{\infty} (P) $, $\mathcal{A} ^{res} _{\Sigma
_{\infty}} $ is the trivial connection, (its
curvature is undetectable)  but $Fuk _{\infty} (P) $
is non-trivial as a fibration so that the aforementioned holomorphic curves and consequently curvature must
appear elsewhere.

\subsection {First quantum obstruction and smooth
invariants} \label{sec:firstquantumobstruction}
We present
here a construction of an integer valued smooth invariant which
is based on our theory. This is probably just the
beginning of the story for invariants of smooth
manifolds based on Floer-Fukaya theory. 

First we discuss a more general setup.
Let $M \hookrightarrow P \xrightarrow{p} X$ be a Hamiltonian $M$-bundle, as
previously. Let $$\operatorname {Fuk} _{\infty} (P)
\xrightarrow{p _{\bullet }} X _{\bullet }$$
be the associated Cartesian fibration, and let  $$K
(P) \to X _{\bullet}$$ be its maximal Kan sub-fibration as in
Lemma \ref{lemma:kanfib}. Then $|K (P)| \to X$ is
a Serre fibration, where $|K (P) |$ is the geometric
realization.

Define 
\begin{equation*} \qobs (P) \in \mathbb{N}
\sqcup \{\infty\}, 
\end{equation*} to be the degree of the first obstruction to
a section of $|K (P)|$. That is $\qobs (P)$ is the smallest integer $n$ such that there is no section of $|K (P)|$ over
the $n$ skeleton of $X$, with respect to some chosen 
CW structure. This is independent of the choice
of the CW structure, as any pair of CW structures on $X$ are
filtered (using cellular filtration)  homotopy equivalent up
to a wedge sum with some collection of $D ^{n}$, $n \in \mathbb{N} $  (with its canonical CW structure),
see Faria~\cite[Theorem 2.4]{cite_FilteredHomotopyTypeFaria}.

When no such $n$ exists we set $\qobs (P) = \infty$.
\begin{theorem}
   \label{thm:firstquantumobstruction} Let $S ^{2}
   \hookrightarrow P \to S ^{4} $ be a non-trivial
   Hamiltonian fibration then:
\begin{equation*}
   \qobs (P)= 4. 
\end{equation*} 
\end{theorem}
In fact we prove that the 
obstruction class in $$H ^{4} (S ^{4}, \pi _{3}
(\operatorname {NFuk} (S ^{2})) $$ is non-trivial.
\subsubsection {First quantum obstruction as a
manifold invariant} \label{sec:manifoldinvariant}
Let $X$ be a smooth manifold, and let $P (X)$
denote the fiber-wise projectivization of $TX
\otimes \mathbb{C}$. 
We then define 
\begin{equation*}
   \qobs (X):= \qobs (P (X)) \in \mathbb{N} \sqcup \{\infty\},
\end{equation*}
which is then an invariant of the smooth manifold
$X$. 
It should be noted that this ``first
quantum obstruction'' invariant is only sensitive
to the tangent bundle, whereas for example
Donaldson invariants can see finer aspects of the
smooth structure. 
In fact the ``quantum Novikov
conjecture'' of Part I would immediately imply
that the first quantum obstruction is only a
topological invariant of $X$, but in that case it is very likely a new
topological invariant (not expressible in terms like
Pontryagin numbers).
\section {Acknowledgements} I am grateful to RIMS institute
at Kyoto university and Kaoru Ono for the invitation,
financial assistance and a number of discussions which took
place there. Much thanks also ICMAT Madrid and Fran Presas
for providing financial assistance, and a lovely research
environment during my stay there. I have also benefited from 
conversations with (in no particular order) Hiro Lee Tanaka,
Mohammed Abouzaid, Kevin Costello, Bertrand Toen and Nick
Sheridan.
\tableofcontents
\section{Preliminaries} \label{sec_preliminaries}
We quickly recall some notation from Part I and some other
basics for convenience.
We denote by $\D$ the simplex category:
\begin{itemize}
	\item The set of objects of $\Delta^{} $ is $\mathbb{N}$.
	\item $\hom _{\D}
 (n, m) $ is the set of non-decreasing maps $[n] \to [m]$, where $[n]= \{0,1, \ldots, n\}$, with its natural order.
\end{itemize} 
A simplicial set $X$ is a functor                                 
\begin{equation*}
   X: \D ^{op}  \to Set.
\end{equation*} 
The set $X (n)$ is called the set of $n$-simplices of $X$.
$\D ^{d} _{\bullet}  $ will denote a particular
simplicial set: the standard representable
$d$-simplex, with $$\D ^{d} _{\bullet} (n) = hom _{\D} (n, d).                                         
$$ 

We use notation $\Delta ^{n} $ to denote the standard topological $n$-simplex.
\begin{definition} \label{definition_simplexcategory} For
$Y$ smooth manifold, $\Delta (Y)$ will denote the
\textbf{\emph{smooth simplex category of $Y$}}. This is the
category s.t.: 
\begin{itemize}
	\item The set of objects $\obj \Delta (Y) $  is the set
of smooth maps:
$$\Sigma: \Delta ^{d} \to X, \quad d \geq
   0. $$
	 \item Morphisms $f: \Sigma _{1} \to \Sigma
      _{2}$ are commutative diagrams in $s-Set$:
\begin{equation} \label{eq:overmorphisms1}  
\begin{tikzcd}
\D ^{d}   \ar[r, ""] \ar [dr,
   "\Sigma _{1}"] &  \D ^{n}  \ar
   [d,"\Sigma _{2}"], \\
    & X.
\end{tikzcd}
\end{equation}
where $\Delta ^{n} \to \Delta ^{m}  $  is a simplicial
map, that is an affine map taking vertices to vertices, preserving the order.
\end{itemize}
\end{definition}
When $Y$ is a smooth manifold $Y _{\bullet}
$ will denote the smooth singular set of $X$. That is $X
_{\bullet } (n)   $ is the set of smooth maps  $\Delta^{n}
\to X$.   

If $p: X \to Y$ is a map of spaces, $p _{\bullet}: X _{\bullet} \to Y _{\bullet}  $ will mean the induced simplicial map. 
We will denote abstract Kan complexes or
quasi-categories by calligraphic letters e.g.
$\mathcal{X}, \mathcal{Y}$.

\section{Outline} \label{section.application} 
In what follows, when we say Part I we shall mean
\cite{cite_SavelyevGlobalFukayaCategoryI}.
We will mostly follow the notation and setup of Part
I, as well as some setup of
~\cite{cite_SavelyevQuantumMaslov}. The reader may review the basics of simplicial
sets, in Section 3 of Part I. For a more detailed introduction, which also includes some theory of
quasi-categories, we recommend Riehl
~\cite{cite_RiehlAmodelstructureforquasi-categories}.    Here are some specific summary points.
Let us briefly review what we do in Part I. Let $M
\hookrightarrow P \xrightarrow{p} X$ be a Hamiltonian
fibration (with monotone compact fibers).

As in Part I, an auxiliary perturbation data $\mathcal{D}$
for $P$, (in particular) involves: 
\begin{itemize}
	\item A choice of a natural system $\mathcal{U} $, see
	Section \ref{sec:constructionSmallData}, consisting of
	certain maps
	\begin{equation*}
{u}: {\mathcal {S}} ^{\circ}  _{d}  \to \Delta
^{n}, \text{with $d,n$ varying},
\end{equation*} as already discussed in Section
\ref{sec:non-metric} of the Introduction.
	\item Choices of certain Hamiltonian connections, on
	Hamiltonian bundles associated to the maps $u$ above.
  We partially review this in Sections \ref{sec_reviewHam},
	\ref{sec:constructionSmallData}.
\end{itemize}
Let $A _{\infty}Cat ^{\mathbb{Z} _{2}} _{k}$ denote the category
of cohomologically unital $A _{\infty}$ categories over
a commutative ring $k$, with $\mathbb{Z} _{2}$ graded hom
complexes, with morphisms $\mathbb{Z} _{2}$ graded $A
_{\infty}$ functors.
Given  $\mathcal{D} $ as above, in Part I we constructed
a functor $$F _{P,k}: \Delta (X) \to A _{\infty}Cat ^{\mathbb{Z}
_{2}} _{k}.$$ This functor
has certain additional properties. 
 
The functor $F _{P,k}$ takes all morphisms to quasi-isomorphisms. 
Moreover, there is an
algebraically induced functor 
$$F ^{unit} _{P,k}: \Delta (X) \to  wA _{\infty}Cat ^{\mathbb{Z}
_{2}} _{k},  $$ 
where $wA _{\infty}Cat ^{\mathbb{Z} _{2}} _{\mathbb{Q}}
\subset A _{\infty}Cat ^{\mathbb{Z} _{2}}$ denotes the
sub-category consisting of strictly unital $A _{\infty}$
categories, with morphisms unital quasi-isomorphisms. 
$F ^{unit} _{P,k}$ is constructed by taking unital
replacements. 
\begin{notation}
\label{not_} In what follows we rename $F _{P,k}$
by $F ^{raw} _{P} $ (raw as in having the raw data of
holomorphic curves) and $F ^{unit} _{P,k} $ by $F _{P}$,
with $k$ implicit. Then the notation matches the one in the
Introduction.

\end{notation}

Let $N$ be the $A _{\infty}$ nerve functor of
Faonte-Lurie as mentioned in the Introduction. 
Taking $k=\mathbb{Q} $ we defined in Part I: 
\begin{equation} \label{eq_FukinftyStrict}
\operatorname {Fuk} _{\infty} (P) = \colim
_{\Delta (X)} N \circ F _{P}. 
\end{equation}
The latter is shown in Part I to be an
$\infty$-category, and moreover it has the structure of a Cartesian fibration:
\begin{equation*}
\operatorname {NFuk} (M,  \omega)
\hookrightarrow \operatorname {Fuk} _{\infty} (P)
\xrightarrow{p _{\bullet }} X _{\bullet}.
\end{equation*}
The equivalence class of the latter under
concordance, see Definition \ref{def:concordancefibration},
is independent of all choices. Note that this is also a property of $F _{P}$. 

In general, given a functor $$F: \Delta^{} (X) \to
wA _{\infty}Cat ^{\mathbb{Z} _{2}} _{\mathbb{Q}},$$
\begin{equation} \label{eq_Fukinfty}
\hocolim
_{\Delta (X)} N \circ F, 
\end{equation}
is a Cartesian fibration over $X _{\bullet}$. Or more
properly stated, there is a representation of this homotopy
colimit that is naturally a Cartesian fibration over the
classical nerve $N(\Delta^{} (X)) $. The latter nerve is
weakly equivalent to $X _{\bullet }$. To get such
a representation we can use the ``Grothendieck
construction'' see
Lurie~\cite[Definition
A.3.5.11]{cite_LurieHighertopostheory}.  

Now for $F
_{P}$ as above $$\hocolim _{\Delta (X)} N \circ F _{P} \simeq
Fuk _{\infty } (P)$$ as Cartesian fibrations. This is again
a property of $F _{P}$, (it is a fibrancy type condition).
The latter assertion is not proved in Part I, but readily
follows the setup of Part I and basic theory of colimits of
simplicial sets.

\subsection{From a Cartesian fibration to a Kan fibration} \label{sec_From a Cartesian fibration to a Kan fibration}
We will
extract from $Fuk _{\infty } (P)$ a Kan fibration and work with that, since then we can just use standard tools of topology.

To this end we have the following elementary lemma.
\begin{lemma} \label{lemma:kanfib}
Suppose we have a Cartesian fibration $p:
   \mathcal{Y}
\to \mathcal{X}$, where $\mathcal{X} $ is a Kan complex. Let
   $K (\mathcal{Y})$
denote the maximal Kan sub-complex of $\mathcal{Y} $ then $p:
K (\mathcal{Y}) \to \mathcal{X}$ is a Kan fibration.
\end{lemma}
The proof is given in Appendix \ref{appendix:Kan}. 
In particular by the above lemma $$K (P):=K (\operatorname
{Fuk} _{\infty}(P)) \xrightarrow{p _{\bullet }} X _{\bullet }$$ is a Kan fibration. 
\begin{notation}
\label{notation:} In what follows $p _{\bullet }$ will
refer to this projection unless specified otherwise.
\end{notation}

\begin{definition} \label{def:concordancefibration}
We say that a Kan fibration, or a Cartesian
fibration $\mathcal{P} $ over a Kan complex
$\mathcal{X} $ is \textbf{\emph{non-trivial}}
if it is not null-concordant. Here $\mathcal{P}
$ is \textbf{\emph{null-concordant}} means that
there is a Kan respectively Cartesian
fibration $$\mathcal{Y}  \to \mathcal{X}
\times \Delta ^{1} _{\bullet} ,$$ whose
pull-back by $i _{0}: \mathcal{X}   \to
\mathcal{X}  \times \Delta ^{1} _{\bullet}   $
is trivial and by $i _{1}: \mathcal{X}  \to
\mathcal{X}  \times \Delta ^{1} _{\bullet}   $
is $\mathcal{P} $. Here the two maps $i_{0}, i_{1}$ correspond to the two vertex inclusions $\Delta ^{0} _{\bullet}  \to \Delta ^{1} _{\bullet}   $.
\end{definition}
\begin{theorem} \label{thm:noSection}
   Suppose that $p: P \to S ^{4} $ is a non-trivial Hamiltonian
	 $S ^{2} $ fibration then $p _{\bullet}: K (P) \to S ^{4}
	 _{\bullet}  $ does not admit a section. In particular $K
	 (P)$ is a non-trivial Kan fibration over $S ^{4}
	 _{\bullet}  $ and so $Fuk _{\infty} (P) $ is
	 a non-trivial Cartesian fibration over $S ^{4} _{\bullet}  $.
\end{theorem} 
This is the main technical result of the paper.
Although in a sense we just are just deducing
existence of a certain holomorphic curve, for this
deduction we need a global compatibility condition
involving multiple moduli spaces, involved in
multiple local datum's of Fukaya categories, so
that this computation will not be straightforward.

\begin{definition}\label{def_concordancefunctor}
Let $I= [0,1]$. We say that two functors $$F _{0}, F _{1}: \Delta^{} (X) \to
w A _{\infty}Cat ^{\mathbb{Z} _{2}} _{k}$$ are
\textbf{\emph{concordant}}  if there is a functor
$$\widetilde{F}: \Delta^{} (X \times I) \to w A _{\infty}Cat
^{\mathbb{Z} _{2}} _{k}$$ s.t.
$\widetilde{F}| _{\Delta^{} (Y \times \{0\})} = F _{0} $ and $\widetilde{F}| _{\Delta^{}
(Y \times \{1\})} = F _{1} $. 
\end{definition}

\begin{proof} [Proof of Theorem \ref{thm:main}]
Suppose otherwise that we have a concordance $\widetilde{F}
$ of $F _{\mathbb{Z} }$ to a constant functor $Const$. 
Then tensoring (all the relevant chain complexes) with
$\mathbb{Q} $ we get a null-concordance of $F _{\mathbb{Q}
}$. Applying the construction \eqref{eq_Fukinfty} we get
a null-concordance of the Cartesian fibration $Fuk _{\infty} (P)$, for $P$ as in
Theorem \ref{thm:noSection}. But this is
a contradiction to the latter theorem. 

The second part of the theorem is deduced as follows.
Suppose by contradiction that $P \not \simeq P'$ and $F
_{P}$ is concordant to $F _{P'}$ so that 
\begin{equation} \label{eq_concordance}
Fuk _{\infty } (P
) \sim Fuk _{\infty } (
{P'}), 
\end{equation}
where $\sim$ denotes the concordance relation.

Now $P = P _{g}$ and $P'
= P _{g'}$, where the right hand side is as in Section
\ref{sec_Ar}. That is $g,g'$ are the corresponding clutching
maps in $\pi _{3} \operatorname {Ham} (S ^{2}) $, and $g
\neq g'$.

Let $\#$ denote the natural bundle connect sum operation.  We have: 
\begin{equation*}
\begin{split}
Fuk _{\infty } (P _{g ^{-1} \cdot g'}) & \simeq Fuk _{\infty} (P
_{g ^{-1}}) \# Fuk _{\infty } (P _{g'}) \text{ by basic
topology and definitions} \\
& \sim Fuk _{\infty} (P _{g ^{-1}}) \# Fuk _{\infty } (P
_{g}) \text{ by \eqref{eq_concordance}} \\
& \simeq Fuk _{\infty } (P _{g ^{-1} \cdot g}).
\end{split}
\end{equation*}
So we get that $Fuk _{\infty } (P _{g ^{-1} \cdot g'})$
is null-concordant. But $P _{g ^{-1} \cdot g'}$ is by assumption a non-trivial
bundle, so that we get a contradiction to the first part of
the theorem.
\end{proof}
\begin{proof} [Proof of Theorem \ref{cor_ktheory}] We can only sketch this, as
there is much prerequisite algebra that is beyond the
scope here. Recall that a finite dimensional complex vector bundle over $S ^{n}$, induces
a non-trivial class in $\pi _{n} (KU \simeq BU \times
\mathbb{Z} ) $ if it has a non-trivial Chern class. The
proof of the corollary is analogous. 

Note that $\pi _{4} (\operatorname {BHam} (S ^{2}, \omega))
\simeq \pi _{3} (SO (3)) = \mathbb{Z}  $, by Smale's
theorem. By 
~\cite[Corollary 6.2]{cite_SavelyevAlgKtheory} we get a homomorphism:
\begin{equation*}
\mathbb{Z} \simeq \pi _{4} (\operatorname {BHam} (S ^{2},
\omega)) \to K ^{Cat, \mathbb{Z} _{2}} _{4} (k).
\end{equation*}

The proof of Theorem \ref{thm:main} is
implicitly by a computing a 
certain characteristic class associated to
a functor of type $F _{P,\mathbb{Q} }$. \footnote {This is
something like the quantum Maslov class also
appearing here but defined algebraically in the
context of a general functor $F: \Delta^{}(X) \to
wA _{\infty}Cat _{k} ^{\mathbb{Z} _{2}} $.}  And these
characteristic classes determine $P$ up to isomorphism. 
In a future work we will make these characteristic  
classes explicit.   

Given this, one proceeds as in the classical case of complex
vector bundles, as recalled above.
\end{proof}

The proof of Theorem \ref{thm:noSection} will be aided by constructing suitable
perturbation data, and will be split into a number of
sections.  


\section{Qualitative description of the perturbation data}
\label{sec_qualitativedescription}
We are now going to discuss the perturbation data needed for
the construction of the functor $F _{P}: \Delta^{} (S ^{4})
\to wA _{\infty}Cat _{\mathbb{Q} } ^{\mathbb{Z} _{2}}$, where
$P$ from now on is as in Theorem \ref{thm:noSection}.

Let $\operatorname {Fuk} (S ^{2}  )$ denote the $\mathbb{Z} _{2} $ graded $A _{\infty} $ category over $\mathbb{Q}$, with objects oriented spin Lagrangian
submanifolds Hamiltonian isotopic to the equator. Our
particular construction of $\operatorname {Fuk} (M, \omega
)$ is presented in Part I. To briefly outline, as part of
some perturbation data $\mathcal{D} $, for each $L, L'  \in \obj \operatorname {Fuk} (M, \omega) 
$ we have Hamiltonian connections $\mathcal{A}  (L,
L')$ on the trivial fibration $[0,1] \times M \to
[0,1]$.
The homomorphism complex $$\operatorname {hom}_ {\operatorname {Fuk} (M, \omega
)} (L, L' )  $$ is the complex $$CF (L, L', \mathcal{A} (L,
L'))$$ defined as in Section 6.1 of Part I. To paraphrase,
the latter 
is the Floer chain complex generated over
$\mathbb{Q}$ by $\mathcal{A} (L,
L' )$-flat sections of $[0,1]
\times M$, with boundary on $L \subset \{0\} \times M,
L'  \subset  \{1\} \times M$.  The homology of $CF (L, L',
\mathcal{A} (L, L'))$ will be denoted by $FH (L, L')$,
understood as a $\mathbb{Z} _{2}$ graded abelian group.

The data $\mathcal{D} $ is generally associated to a Hamiltonian fibration. As a symplectic manifold is a Hamiltonian fibration over a point, we write $\mathcal{D} _{pt}$ for this restricted data, needed for construction of
$\operatorname {Fuk} (S ^{2}  )$ as outlined above.

Denote by $\operatorname {Fuk ^{eq}}  (S ^{2}
) \subset \operatorname {Fuk} (S ^{2}  )$ the full
sub-category obtained by restricting our objects to be
oriented standard equators in $S ^{2} $. We take our perturbation data
$\mathcal{D} _{pt} $ so that the following is satisfied. 
\begin{itemize}
	\item All the connections $\mathcal{A} (L, L') $ as above are $PU
(2)$-connections.  
\item For $L$ intersecting $L'$ transversally, the $PU (2)$ connection $\mathcal{A} (L, L')$ is the trivial flat connection. 
\item For $L=L'$ the corresponding connection is generated by an autonomous Hamiltonian.
\end{itemize}


The associated cohomological Donaldson-Fukaya category $DF  (S ^{2} )$ is equivalent as a linear category over $\mathbb{Q} $ to $FH (L_0, L _{0} )$ (considered as a linear category with one
object) for $L _0 \in \operatorname {Fuk} (S ^{2} ) $.

It is easily verified that a morphism (1-edge) $f$ is an
isomorphism in the nerve $\operatorname {NFuk}(S
^{2} )$, if and only if it corresponds, under the nerve construction $N$, to a morphism in 
$\operatorname {Fuk} (S ^{2} ) $ that induces an isomorphism in $DF (S ^{2})$. Such a morphism will be called a \emph{$c$-isomorphism}.

Consequently the maximal Kan subcomplex $K (S ^{2} )$ of
$\operatorname {NFuk}  (S
^{2} ) $ is characterized as the maximal subcomplex with 1-simplices the images by $N$ of $c$-isomorphisms in $Fuk (S ^{2} )$.

\subsection{Extending $\mathcal{D} _{pt} $ to higher dimensional simplices}
\label{section:model}
\begin{terminology} A bit of possibly non-standard terminology: we say that $A$ is a \emph{model}
for $B$ in some category, with weak equivalences, if there is a morphism
$mod: A \to B$ which is a weak-equivalence.
The map $mod$ will be called a \emph{modelling map}. 
In our
context the modeling map $mod$ always turns out to be a monomorphism.
\end{terminology}
Let us model $D ^{4} _{\bullet}  $ and $S ^{3} _{\bullet} $  as follows. Take the standard representable 3-simplex $\Delta ^{3} _{\bullet}$, and the standard
representable 0-simplex $\Delta ^{0} _{\bullet}  $. Then collapse all faces of $\Delta
^{3}_{\bullet}  $ to a point, that is take the colimit of the following diagram:
\begin{equation}
 \begin{tikzcd}
   & & \Delta ^{0} _{\bullet}     &  \\
  \Delta ^{2} _{\bullet} \ar [drr, "i_0"] \ar [urr] &  \Delta ^{2} _{\bullet} \ar [dr, "i_1"] \ar [ur]   & \Delta ^{2} _{\bullet}
    \ar [d, "i_2"]  \ar [u] & \Delta ^{2}
   _{\bullet} \ar [dl, "i _{3}"]  \ar [ul]   \\
 & &   \Delta ^{3} _{\bullet}   &  \\  
\end{tikzcd}
\end{equation}
Here $i _{j} $ are the inclusion maps of the non-degenerate 2-faces.
This gives a simplicial set $S _{\bullet} ^{3, mod}  $ modeling the simplicial set $S ^{3} _{\bullet} $, in other words there is a natural a weak-equivalence $$S ^{3,mod} _{\bullet} \to S ^{3} _{\bullet}.  $$


Now take the cone on $S _{\bullet} ^{3, mod}  $, denoted by $C (S ^{3,mod}
_{\bullet})$, and collapse the one non-degenerate 1-edge.
The resulting simplicial set $D
^{4,mod}_{\bullet}  $ is our model for $D ^{4} _{\bullet}  $, it may be identified with a subcomplex of $D ^{4} _{\bullet}  $ so that the inclusion map $mod: D
^{4,mod}_{\bullet} \to D ^{4} _{\bullet}   $ induces a weak homotopy equivalence of pairs 
\begin{equation} \label{eq:homotopyequivalence}
(D^{4,mod}_{\bullet}, S ^{3,mod}  _{\bullet} ) \to  (D ^{4} _{\bullet}, S ^{3}  _{\bullet}). 
\end{equation}
We set $b _{0} \in D ^{4} _{\bullet} $ to be the vertex which is the image by $mod$ of the unique 0-vertex in $D ^{4,mod} _{\bullet}   $.

Suppose we have a commutative diagram:
\begin{equation*}
\begin {tikzcd}
    D ^{4}     \ar[r,"h_+"] &  S ^{4}  &  \arrow {l} [above] {h_-} D ^{4} \\
    & S ^{3} \ar [ul, "i"] \arrow {ur} [below] {i} &  
  \end{tikzcd}
\end{equation*}
where $i: S ^{3} \to D ^{4} $ is the natural boundary
inclusion, and s.t. the following is satisfied.
\begin{itemize}
	\item $h _{\pm}:D ^{4} \to S ^{4} $ are smooth, and their
	images cover $S ^{4} $.
	\item $$h _+ ({D} ^{4})  \cap h _- ({D ^{4} })  $$ is contained in the image $E$ of $$h _{\pm} \circ i: S^{3}  \to S ^{4}. $$
	\item $h _{\pm} $ takes $b _{0} $ to $x _{0} $.
\end{itemize}


For example, we may just let $h _{-}$ represent the
generator of $\pi _{4} (S ^{4}, x _{0}) $ and $h _{+} $ to
be the constant map to $x _{0} $. 
We call such a pair $h _{\pm} $ a \emph{complementary pair}.

We set $$D _{\pm}:= h _{\pm} (D ^{4,mod} _{\bullet}) \subset
S ^{4} _{\bullet},   $$ and we set $\Sigma _{\pm} \in S ^{4} _{\bullet} $ to be the image by $h _{\pm} $ of the sole non-degenerate 4-simplex of $D ^{4,mod} _{\bullet}  $. Also set $$\partial D _{\pm} := h _{\pm}(\partial D ^{4,mod} _{\bullet}), $$ where $\partial D ^{4,mod} _{\bullet}$ is the image of the natural inclusion $S ^{3,mod} _{\bullet} \to  D ^{4,mod} _{\bullet}$.

Fix a Hamiltonian frame for the fiber $P_{x_{0}} $ of $P$
over $x _{0} $, in other words a Hamiltonian bundle
diffeomorphism 
\begin{equation*}
\begin{tikzcd}
S ^{2} \ar[r, ""] \ar [d, ""] &  P \ar [d,"p"] \\
pt \ar [r, "x _{0}"]   & S ^{4}.
\end{tikzcd}.
\end{equation*}
In particular, this allows us to identify
$\operatorname {Fuk} (S
^{2}  )$ with $F ^{raw} (x _{0})$, using the
analytic perturbation data
$\mathcal{D} _{pt}$ for both. Denote by $x
_{0,\bullet}$ the image of the map $$\Delta ^{0} _{\bullet}
\to S ^{4,mod} _{\bullet},  $$ induced by the inclusion of
the 0-simplex $x _{0}$. 

We continue with the description of the data $\mathcal{D}
= \mathcal{D} (P) $. This must associate certain data
$\mathcal{D} _{\Sigma}$ for each singular simplex $\Sigma: \Delta^{n}
\to S ^{4}$. Now in Section 8 of Part I, given the data
$\mathcal{D} _{\Sigma}$  for a non-degenerate simplex
$\Sigma $, we assigned extended perturbation data
$\mathcal{D} _{\widetilde{\Sigma}}$  for all
degeneracies $\widetilde{\Sigma} $ of this simplex. 
So by that discussion, our chosen data $\mathcal{D} _{pt} $ induces perturbation data for all degeneracies of $x _{0} $, that is 
for all simplices of $x _{0,\bullet}$, this data will again
be denoted by $\mathcal{D} _{pt} $, for simplicity.

Fix an object 
\begin{equation} \label{eq_L0}
L _{0} \in \operatorname {Fuk ^{eq}} (S ^{2} ) \subset
F ^{raw} (x _{0}).
\end{equation}
Denote by $\gamma \in \hom _{F ^{raw}  (x_0)} (L _{0}, L _{0}) $ the
generator of $FH _{1}  (L_0, L _{0}) $,  i.e. the fundamental chain, so that it corresponds to the identity in $DF  (L _{0}, L _{0}  )$. This $\gamma$ is uniquely determined by our conditions and corresponds to a single geometric section. Denote by $L ^{i}  _{0} $ the image of $L _{0} $ by the embedding $$F ^{raw}  (x
_{0}) \to F ^{raw} (\Sigma _{+}),$$ corresponding to the $i$'th vertex inclusion into
$\Delta ^{4} $, $i=0, \ldots ,4$. 

Let $m _{i} $ be the edge between $i-1,i$ vertices and set $$\overline{m}_{i}:=\Sigma _{+}  \circ m _{i}.$$ Let $\Sigma ^{0} _{i}  $ denote the
0-simplex obtained by restriction of $\Sigma ^{4} $ to the $i$'th vertex.
Note that each $\overline {m} _{i} $ is degenerate by
construction, so we have an induced morphism $$F ^{raw}
(pr): F ^{raw}  (\overline{m} _{i}) \to F ^{raw}  (x _{0} ),$$
for $pr$ the degeneracy morphism in $\Delta (S ^{4})$:  $$pr: \overline{m} _{i}  \to \Sigma ^{0} _{i}.    $$
Finally, for each $L _{0} ^{i-1}, L _{0} ^{i} $ we have
a $c$-isomorphism $$\gamma _{i}: L_{0} ^{i-1} \to L _{0}
^{i}    $$ in $F ^{raw}  (\overline{m} _{i}) \subset
F ^{raw} (\Sigma _{+} ) $, which corresponds to $\gamma$, meaning that the
fully-faithful projection $F ^{raw}  (pr)$ takes
$\gamma _{i} $ to $\gamma$. We will denote by $\gamma _{i,j} $ the analogous
$c$-isomorphisms $L _{0} ^{i} \to L _{0} ^{j}    $.

\begin{notation}
 Let us denote from now on, the morphism spaces  $hom _{F
 ^{raw} (\Sigma  _{\pm})} (L _{0}, L _{1}) $ by $hom
 _{\Sigma  _{\pm}} (L _{0}, L _{1}) $. And
 denote the $A _{\infty}$ composition maps $\mu ^{d} _{F
 ^{raw} (\Sigma _{\pm})}$, in the $A _{\infty }$ category $F
 ^{raw} (\Sigma _{\pm}) $, by $\mu ^{d} _{\Sigma _{\pm}}$. 
\end{notation}
\begin{definition} \label{def:smalldata} We call
perturbation data $\mathcal{D}$ for $P$
\textbf{\emph{unexcited}} if it is extends the data $\mathcal{D} _{pt} $ as above, and if with respect to $\mathcal{D}$ 
 \begin{equation} \label{eq:gammaunit+}
\mu ^{d} _{\Sigma _{+}} (\gamma ^{1}, \ldots,
   \gamma ^{d} )=0, \text{ for $2 <d <4$},  
\end{equation}
   where $(\gamma ^{1}, \ldots, \gamma ^{d}) $ is
	 a composable chain, and each $\gamma ^{k} $ is of the
	 form $\gamma _{i,j} $  as above.
\end{definition}
%
We will see further on how to construct such unexcited data, assume for now that it is constructed.

Let $\{f _{J} \} $,  corresponding to an $n$-simplex, be as
in the definition of the $A _{\infty} $ nerve in Appendix
A.4 Part I, where $J$ is a subset of $[n] = \{0, \ldots, n\}$.

\begin{lemma} \label{lemma:conditionsfJ}
 Let $\mathcal{D}$ be unexcited as above, then there is a 4-simplex $\sigma \in NF ^{raw}(\Sigma _{+})$ with faces determined by the conditions:
\begin{itemize} 
   \item $f _{J} =0$, for $J$ any subset of $[4]$ with at least $3$ elements.
   \item $f _{\{i-1,i\}} = \gamma _{i}$ for $\gamma _{i} $ as before.
\end{itemize}  
\end{lemma}
\begin{proof} 
This follows by \eqref{eq:gammaunit+} and by the identity $\mu ^{2} _{\Sigma _{+}  } (\gamma, \gamma) = \gamma $.
   \end{proof}
If we take our unital replacements so that $\gamma$
corresponds to the unit, then $\sigma$ induces (by the
construction) a section of $K (P _{+}) \to D _{+} $, where $K (P _{\pm} )$ will be shorthand for $K (P)$ restricted over ${D _{\pm}}$.

%
Let
\begin{equation*}
   i: \left(K (P _{+})|_{\partial D_ {+}}:= p _{\bullet} ^{-1} (\partial D_ {+}) \right)   \to K (P_{-}),
\end{equation*}
be the natural inclusion map. 
Set $$sec = i
\circ \sigma \circ h _{+}| _{\partial D ^{4,mod} _{\bullet}   }.  $$


\subsection {The main lemma and immediate consequences} \label{section:immediateconsequences}
\begin{lemma} \label{lemma:nontrivialelement} 
Suppose that $P$ is a non-trivial Hamiltonian fibration and
$\mathcal{D}$ is unexcited data for $P$ as above, then $sec$ as
above does not extend to a section of $K (P _{-})$.
Moreover, unexcited data $\mathcal{D} $ exists.
\end{lemma} 
This lemma involves all the ingredients of our theory,  its
proof that will be broken up in parts, and will follow shortly.
\begin{proof} [Proof of Theorem \ref{thm:firstquantumobstruction}]
Clearly $\qobs (P) \geq 4$, since the $3$-skeleton of $S ^{4} $ is trivial. 
By Lemma \ref{lemma:nontrivialelement} above, $K (P)$ does
not have a section over the $4$-skeleton. 
\end{proof}
\begin{remark}
 When $P$ is obtained by clutching with a generator of $\pi _{3} (PU (2)) $, and when $h _{\pm} $ are embeddings, the class $[sec]$ in $\pi _{3} (K (P _{-} )) \simeq \pi _{3} (K (S ^{2} )) $ can be thought of as ``quantum'' analogue of the class of the classical Hopf map.
\end{remark}
\begin{proof}[Proof of Theorem \ref{thm:noSection}]
It might be helpful to first review Appendix
\ref{appendix:Kan} before reading the following.
If we take any unexcited perturbation data $\mathcal{D}$ for $P$, then the first part follows immediately by Lemma \ref{lemma:nontrivialelement}.
So $K (P)$ is non-trivial as a Kan fibration.

This then implies that $\operatorname {Fuk} _{\infty} (P)$ is non-trivial as
a Cartesian fibration. 
To see this, suppose otherwise that we have 
a Cartesian fibration
$$\widetilde{\mathcal{P} } \to S ^{4}
_{\bullet}  \times I _{\bullet},  $$
restricting to $\operatorname {Fuk} _{\infty} (P) $ over $S
^{4} _{\bullet}  \times 0 _{\bullet}  $ and to
$\operatorname {NFuk} (S ^{2} ) \times S ^{4} _{\bullet}   $
over the other end $S ^{4} _{\bullet}  \times 1
_{\bullet}  $. 
Here $0 _{\bullet} $,
respectively $1
_{\bullet} $ are notation for the  images of $i _{j, \bullet }: \Delta ^{0} _{\bullet}
\to \Delta^{1}  _{\bullet }$, $j=0,1$, where $i
_{j, \bullet }$ are induced by the pair of
boundary point inclusions.  

Now take the maximal Kan sub-fibration of $\widetilde{\mathcal{P} } $, then by Lemma \ref{lemma:kanfib} we obtain a trivialization of $K (P)$ which is a contradiction. 
\end{proof}

\section {Towards the proof of Lemma \ref{lemma:nontrivialelement}} \label{section:partI}

We will denote by $L _{0, \bullet} $ the image of the map $\Delta ^{0} _{\bullet} \to K (P _{-}),$ induced by the inclusion of $L _{0}$ into $K (S ^{2} )$ as a 0-simplex.  Suppose that $sec$ extends to a section of $K (P _{-})$, so we have map $$e: D ^{4,mod} _{\bullet} \to K (P_{-}) $$ extending $sec$ over $\partial D ^{4,mod} _{\bullet}  $. 
We may assume WLOG that $e$ lies over $h_{-}$, meaning $$p
_{\bullet} \circ e = h _{-},$$ since it can be homotoped to
have this property. To see the latter, first take a relative homotopy of 
\begin{equation*}
   p _{\bullet} \circ e:  (D ^{4,mod} _{\bullet}, \partial D ^{4,mod} _{\bullet})   \to (D _{-}, \partial D _{-})
\end{equation*}
to $h _{-} $, using that we have a homotopy equivalence of pairs \eqref{eq:homotopyequivalence},
and then lift the homotopy to a relative homotopy upstairs using the defining lifting property of Kan fibrations.

And so we have a 4-simplex $$T = e (\Sigma ^{4}) \in K (P
_{-} ) $$ projecting to $\Sigma _{-} \in D _{-} $ by $p
_{\bullet} $. Since $T$ is in the image of $e$, all but one
3-faces of $T$ are totally degenerate with image in $L _{0,
\bullet } $. The exceptional 3-face is the sole
non-degenerate 3-face of $sec$, (that is of $sec (\partial D ^{4,mod} _{\bullet})$).

Let $m _{i,j}, \gamma _{i,j} $ be as in the previous
section, but corresponding now to $\Sigma _{-} $ rather than
$\Sigma _{+} $. Then by the boundary condition on $e$, the edges
of $T$ (which are all edges of $sec$)  correspond, under the
nerve construction, to the
generators $\gamma _{i,j}$. As this is the condition for the
edges of $sec$.

\begin{lemma} \label{lemma:corellator}
For $\mathcal{D}$ unexcited, and for the unital replacement $F$ of $F ^{raw} $ as above, the simplex $T$ exists if and only if
$\mu ^{4} _{\Sigma _{-}} (\gamma _{1},  \ldots,  \gamma _{4} )$
 is exact. 
\end{lemma}
\begin{proof} 
 The following argument will be over
   $\mathbb{F}_{2} $ as opposed to $\mathbb{Q}$ as the signs will not matter. Recall that we take the unital replacement so that $\gamma \in hom_ {F^{raw} (P _{x _{0}})} (L _{0}, L _{0})$ corresponds to the unit in the unital replacement.

Now if $T \in K (P _{-} )$ as above exists, then it
corresponds under unital replacement (see Remark 7.5 in Part
I) to a $4$-simplex $T' \in NF ^{raw} (\Sigma _{-})$
satisfying the following condition on its $4$-face.
Recalling the nerve construction, the morphism $f _{[4]} \in hom _{\Sigma  _{-}} (L_{0} ^{0},
L _{0} ^{4} ) $, figuring in the definition of the $4$-face,
satisfies:
\begin{equation} \label{eq:simplex}
\mu ^{1} _{\Sigma  _{-}} f _{[4]} = \sum _{1 < i < 4} f _{[4] - i} + \sum _{s} \sum _{(J _{1},
   \ldots J _{s} ) \in decomp
_{s} } \mu ^{s} _{\Sigma  _{-} }  (f_{J _{1}}, \ldots , f _{J _{s} }  ).
\end{equation}
By our conditions on the boundary of $T$, by the condition on the unital replacement, and by the conditions in Lemma \ref{lemma:conditionsfJ}, 
we must have $f_{J} =0 $, for every proper subset $J \subset [4]$, in some length $s$ decomposition of $[4]$, unless $J=\{i,j\}$ in which case  $f _{i,j}=\gamma _{i,j} $.
Given this \eqref{eq:simplex} holds if and only if $\mu ^{4} _{\Sigma _{-}}
(\gamma _{1} , \ldots , \gamma _{4} ) $ is exact. 

\end{proof}
We are going to show that for some unexcited $\mathcal{D}$,   $ \mu ^{4} _{\Sigma _{-}} (\gamma _{1} ,  \ldots,  \gamma
_{4} )$  does not vanish in homology, which will finish
the proof of the Lemma \ref{lemma:nontrivialelement}. 
\section{Review of Hamiltonian structures}
\label{sec_reviewHam}
We review here the notion of a Hamiltonian
structure. A complete formal definition of this is contained
in ~\cite[Definition
\ref{def:HamiltonianStructure}]{cite_SavelyevQuantumMaslov}.
A \emph{Hamiltonian structure}  $$\{\widetilde{{S} } _{r}, {S} _{r},
\mathcal{L} _{r}, \mathcal{A} _{r} \}
_{\mathcal{K} }, 
$$
consists of the following data.
\begin{itemize}
\item A compact smooth manifold $\mathcal{K} $ possibly with
boundary and corners (in the latter case $\mathcal{K} $ is
meant to be a polyhedron so that there is no need for any
deep theory of manifolds with corners).
	\item A suitably smooth $r \in \mathcal{K} $ family of Hamiltonian
fibrations $M \hookrightarrow \widetilde{S}  _{r} \to S _{r} $, where $S _{r}$
are Riemann surfaces with boundary and with ends, s.t. each
$S _{r}$ is diffeomorphic to a closed disk with some number
of punctures on the boundary. The data of
a ($r$-dependent) holomorphic
diffeomorphism at each $i$'th end, $$e _{i}:  [0,1] \times (0,
  \infty) \to S _{r},$$ $i \neq 0$, (called positive ends). At the 0'th (called negative) end
	we ask for a holomorphic diffeomorphism $$e _{0}: [0,1] \times (-\infty, 0) \to S _{r}.
  $$  
These charts are called {\emph{strip end charts}}. 
\item The Hamiltonian fibrations $\widetilde{S}  _{r}$ are endowed with
Hamiltonian connections $\mathcal{A} _{r}$, which preserve
a fiberwise Lagrangian subfibration $\mathcal{L} _{r}$ of
$\widetilde{S} _{r}$ over $\partial S _{r}$. That is we have
a fibration $\mathcal{L} _{r} \to \partial S _{r}$, which is
embedded in $\widetilde{S} _{r}$, with fibers embedded as
Lagrangian submanifolds of the fibers of $\widetilde{S}
_{r}$.
\item Trivializations $$\widetilde{e} _{i}: [0,1] \times (0, \infty) 
\times M \to \widetilde{S} _{r}, $$ over $e _{i}$, 
called \emph{end bundle charts}  for $\widetilde{S} _{r}$.
\item In the end bundle charts, the Hamiltonian connections $\mathcal{A} _{r}$ are
flat and translation invariant. Specifically, they have the form $\overline{\mathcal{A}} _{i}$,
where $\overline {\mathcal{A}} _{i}   $ denotes the
$\mathbb{R}$-translation invariant extension of
some Hamiltonian connection $\mathcal{A} _{i} $ on $[0,1]
\times M$ to $(0, \infty) \times [0,1]
\times M$, (for the case of a positive end).
\item A certain choice of a smooth family almost complex
structures $j _{z}$ on the fibers $M _{z}$ of $\widetilde{S}
_{r}$ (also $r$ dependent). However, the latter will be
implicit in the sense that we do not need to manipulate
them, so we do not elaborate here, instead we refer the reader to ~\cite[Definition
\ref{def:HamiltonianStructure}]{cite_SavelyevQuantumMaslov}.
 
\end{itemize}
For $\mathcal{A} _{i}$ as above,  we say that $\mathcal{A}
_{r}$ are \emph{compatible}  with $\{\mathcal{A}
_{i}\}$. 
Let us write $(\widetilde{S}, \mathcal{L} )$ for a pair as
above.

\subsection{The Hamiltonian structures associated to the data
$\mathcal{D}$} \label{sec_The Hamiltonian structures associated to the data
}
As in Section 4 of Part I, $\Pi (\Delta^{n} )$ will denote
a certain fundamental groupoid of $\Delta^{n} $. More
specifically, $\Pi (\Delta ^{n}) $ is the small groupoid, whose
objects set $obj$ is the set of vertices
of $\Delta ^{n}$. The morphisms set $hom$ is
the set of affine maps $m: [0,1] \to \Delta ^{n}$, 
(possibly constant) sending end points to the
vertices. The source map $$s: hom \to obj$$ is defined by
$s (m) = m (0)  $ and the target map $$t: hom \to obj$$ is
defined by $t(m) = m (1)$.

Let $(m ^{1}, \ldots, m ^{d}) $  be a composable chain 
of morphisms in $\Pi (\Delta ^{n}) $, which we recall means
that the target of $m ^{i-1} $ is the source of $m ^{i}
$ for each $i$. 
The perturbation data $\mathcal{D} $, in
particular specifies for each $n$ and for each such
composable chain, certain maps $$ 
u (m ^{1}, \ldots, m ^{d},n): \mathcal{E} _{d} ^{\circ}   \to \Delta ^{n}. 
$$
Here $\mathcal{E} _{d}  $ is the universal curve over
$\overline{\mathcal{R}} _{d}$, and $\mathcal{E} _{d}
^{\circ}  $ denotes $\mathcal{E} _{d}  $ with nodal points
of the fibers removed. The collection of these maps,
satisfying certain axioms, is denoted  by $\mathcal{U}$. We have already mentioned this in the introduction.

The restriction of $u (m ^{1}, \ldots, m ^{d},n)$ to the fiber
$\mathcal{S} _{r} $ of $\mathcal{E} _{r} ^{\circ}  $ over
$r \in \overline{\mathcal{R}} _{d} $, is denoted by $u (m
^{1}, \ldots, m ^{d},n, r)$ which may also be abbreviated by
$u _{r} $.

Let $M \hookrightarrow P \to X$ be a Hamiltonian fibration
with compact monotone fibers.
Let $\Sigma: \Delta ^{n} \to X$ be smooth, and 
denote $\overline {m} ^{i} := \Sigma \circ m ^{i}  $.
Suppose further we are given $u (m ^{1}, \ldots, m ^{d},n)$ and a
chain of Lagrangian branes (objects of the monotone Fukaya
category $Fuk (M, \omega )$) $L' _{0}, \ldots, L' _{d}  $ with
$L' _{i} \subset P| _{\overline{m} _{i} (1)} $, $i \geq
1$, $L' _{0} \subset P | _{\overline{m} _{1} (0)} $.
Then the data $\mathcal{D} $ associates to this a Hamiltonian structure:

\begin{equation} \label{eq_structureMainExample}
\{\widetilde{\mathcal{S} } _{r}, \mathcal{S} _{r},
\mathcal{L} _{r}, \mathcal{A} _{r} \}
_{\overline{\mathcal{R}} _{d}}, 
\end{equation}
where: 
\begin{itemize}
\item $\widetilde{\mathcal{S}} _{r}:= (\Sigma \circ u _{r}) ^{*}P$.
\item $\widetilde{\mathcal{S}} _{r}$ is naturally trivial
over the boundary components of $\mathcal{S} _{r}$.
Likewise $\mathcal{L}  _{r}$ is trivial over each $i$'th component of ${\partial \mathcal{S} _{r}}$, with fiber $L' _{i}$. Here the subscripts $i$
correspond to the labels in the Figure
\ref{fig:sidenumberedsurface}. 
\item $\mathcal{A} _{r}$ are compatible with $\{\mathcal{A}
_{i} = \mathcal{A} (L'
_{i-1}, L' _{i})\} $, with the later part of the data of the
construction of the Fukaya category as in the preamble of
Section \ref{sec_qualitativedescription}.
\end{itemize}
\begin{figure}[h]
  \includegraphics[width=2in]{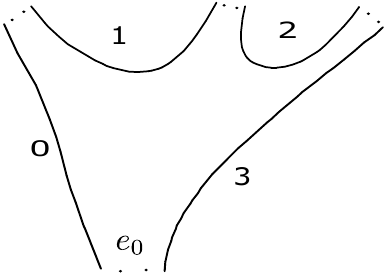}
 \caption {} \label{fig:sidenumberedsurface}
\end{figure}

\section {Construction of unexcited data}
\label{sec:constructionSmallData}

Now, specialize the discussion of the section above to the case $n=4$,  $\Sigma = \Sigma
_{+} $, $\forall i: L' _{i} = L _{0}  $, with $L _{0}$ the
distinguished equator as in \eqref{eq_L0}. In this case, we
write $\mathcal{A} _{r} ^{+} $ for the corresponding
connections, ($m _{1}, \ldots, m _{d}$ will be implicit).
 
Suppose that $\mathcal{D}$ extends $\mathcal{D} _{pt} $ from
before. If $\mathcal{L} _{r} \subset \widetilde{\mathcal{S}}
_{r}| _{\partial \mathcal{S} _{r} }  $ denotes the trivial
Lagrangian sub-bundle with fiber $L _{0} $, then we obtain
a Hamiltonian structure $$\Theta ^{+}= \{\Theta ^{+} _{r}
\}= \{\widetilde{\mathcal{S}} _{r}, \mathcal{S} _{r},
\mathcal{L} _{r}, \mathcal{A} ^{+}   _{r}    \}
_{\overline{\mathcal{R}} _{d}}. $$ 

To recall, as part of the requirements, at each
end $e _{i} $ of $\mathcal{S} _{r} $, $\mathcal{A} ^{+} _{r}
$ is compatible with the connection $\mathcal{A} _{i}
=  {\mathcal{A}} (L _{0}, L _{0})$.   

Set $$\hbar:= \frac{1}{2} \Vol (S ^{2},\omega). $$
Let $\kappa$ denote the $L ^{\pm} $-length of the holonomy
path in $\operatorname {Ham} (S ^{2}, \omega) $ of $\mathcal{A} _{0}= \mathcal{A} (L _{0}, L _{0})$. We may suppose that  
\begin{equation} \label{eqAreaAd}
\forall r: \energy (\mathcal{A} ^{+}  _{r}  ) <  \hbar- 5 \kappa,
\end{equation}
is satisfied after taking $\kappa$ to be sufficiently small.
(There is no obstruction since the corresponding bundles are
naturally trivializeable, continuously in $r$.) 

Fix a complex structure $j _{0} $ on $M$,
and let $\{J _{r} = J (\mathcal{A} ^{+} _{r})  \}$ be the
corresponding, induced family of almost complex structures on
$\{\widetilde{\mathcal{S}} _{r}  \}$ as defined in
~\cite[Section \ref{sec:ModuliSpacesHamStructures}]{cite_SavelyevQuantumMaslov}.

Let $\overline{{\mathcal{M}}} (\Theta ^{+} 
,A)$ be as in \cite[Section \ref{sec_Family
version}]{cite_SavelyevQuantumMaslov}. To paraphrase, this
is the set of pairs $(u,r)$ for $u$ a relative homology class $A$, $J
_{r}$-holomorphic, finite Floer energy section of
$\widetilde{\mathcal{S} } _{r}$, with boundary on
$\mathcal{L} _{r}$. The class $A$, again to paraphrase,  is
a relative class represented by an asymptotically
$\mathcal{A} _{r}$-flat section of $\widetilde{S} _{r}$,
with boundary on $\mathcal{L} _{r}$.

As in Part I,  let $$ 
   \overline {\mathcal{M}}=\overline{{\mathcal{M}}} (\gamma ^{1}, \ldots, \gamma ^{d}; \gamma ^{0}, \Sigma _{+},
\{J_{r}   \},
   A), 
$$ denote the set of elements of $\overline{{\mathcal{M}}} (\Theta ^{+} 
,A)$ with asymptotic constraints $\gamma ^{i}$ at each $e
_{i} $ end. Here each $\gamma ^{k}   $, $k \neq 0$, is of
the form $\gamma _{i,j} $ where this is as in Section \ref{section:model}.
\begin{lemma}
\label{lemma:empty0} 
Whenever the class $A$ is such that ${\mathcal{M}}$ has virtual dimension $0$, and $d$ satisfies $2 < d \leq 4$, $\overline {\mathcal{M}}$ is empty.
\end{lemma}
\begin{proof}
   Let $$\Theta ^{/} := (\Theta ^{+}) ^{/},   $$ and $A ^{/}
	 $ be the capping off construction as 
	 in ~\cite[Section 2.7]{cite_SavelyevQuantumMaslov}. For a fixed $r$,
	 by the Riemann-Roch,
	 ~[Appendix A]\cite{cite_SavelyevQuantumMaslov},  we get that the expected dimension of ${\mathcal{M}} (\Theta ^{/}, A ^{/}) $ 
is 
\begin{equation*}  1+ Maslov ^{vert}  (A ^{/} ).
\end{equation*} 
Consequently, when $\gamma ^{0} =\gamma $, the expected dimension of $\mathcal{M}$ is:
\begin{equation} \label{eqDimension} 1+ Maslov ^{vert}  (A ^{/})  - 1   + (\dim
   \mathcal{R}_{d} = d-2).
\end{equation} 
We need the expected dimension of $\mathcal{M}$ to be 0, and $d \geq 3$, so $Maslov ^{vert} (A ^{/} ) \leq -1$. But $Maslov ^{vert} (A ^{/} )
   = -1$ is impossible as the minimal positive Maslov number
	 is 2.   

Now, note that if $Maslov ^{vert} (A ^{/} ) = -2$ then $$-C \cdot Maslov ^{vert} (A ^{/} )
= \hbar, $$ for $C$ the monotonicity constant of $(S ^{2},
\omega ) $ and the equator $L _{0}$.  Consequently, the
result follows by ~\cite[Lemma
2.34]{cite_SavelyevQuantumMaslov} and by the property \eqref{eqAreaAd}.

When $\gamma ^{0}  $ is the Poincare dual to $\gamma $, we would get $Maslov ^{vert} (A ^{/} )
   \leq -2$  so for the same reason the conclusion follows. 
\end{proof}

So if we choose our data $\mathcal{D}$ so that the hypothesis of the lemma above are satisfied,
then with respect to this $\mathcal{D}$:
\begin{align} 
   \label{eq:mu:identities2} \mu ^{2} _{\Sigma ^{4} _{\pm}    }
(\gamma _{i,j} , \gamma _{j,k} ) &  = \gamma _{i,k}  \\
   \label{eq:mu:identities1}
   \mu ^{3} _{\Sigma ^{4} _{\pm} }  (\gamma ^{1}, \ldots, \gamma ^{3}) & = 0,    \text{ for $\gamma ^{i} $ as above} \\
\mu ^{4} _{\Sigma ^{4} _{+}}  (\gamma _{1}, \ldots, \gamma _{4}) & = 0.
\end{align}
In particular such $\mathcal{D} $ is unexcited, and
$\mathcal{D} $ from now on will denote such a choice. 

\section {The product $\mu ^{4} _{\Sigma ^{4}
_{-}} (\gamma _{1} , \ldots , \gamma _{4} ) $ and
the quantum Maslov classes} \label {section:productAndHigherSeidel}

The product 
\begin{equation}\label{eq_productmu}
  \mu ^{4} _{\Sigma ^{4} _{-}} (\gamma _{1}
, \ldots , \gamma _{4} ),
\end{equation}
a priori depends on various
choices, like the choices of $h _{\pm} $, and then choice of
data $\mathcal{D} $.   
For the purpose of computation we will take 
$h _{+} $ to be the constant map to $x _{0} $ and $$h _{-}:
(D ^{4}, \partial D ^{4}) \to (S ^{4}, x _{0})    $$ to be the complementary map, that is representing the generator of $\pi _{4} (S ^{4}, x _{0}) \simeq \mathbb{Z}$. We further suppose that $h _{-} $   is an embedding in the interior of $D ^{4} $.

We are then going to reduce the computation of the product
\eqref{eq_productmu} to the computation of a certain quantum
Maslov class, already done in
~\cite{cite_SavelyevQuantumMaslov}. This will require
a number of geometric steps.

Let $\Sigma _{-} $ be the 4-simplex of $S ^{4} _{\bullet}
$ corresponding to $h _{-}$  as before in Section \ref{section:model}. We need to study the moduli spaces
\begin{equation} \label{eqModuli}
\overline{\mathcal{\mathcal{M}}} ( \gamma _{1}, \ldots,  \gamma _{4}; \gamma
   ^{0}, \Sigma_{-}, \{\mathcal{A}_{r} \}, A),
\end{equation}
where $\mathcal{A} _{r}$ now denotes the connections on
\begin{equation} 
 \widetilde{\mathcal{S}}   _{r}:= (\Sigma _{-}  \circ u (m _{1}, \ldots,  m _{4}, 4,r)) ^{*}P \to \mathcal{S} _{r},
\end{equation}
part of the data $\mathcal{D} $. We abbreviate $u (m _{1}, \ldots,  m _{4}, 4,r)$ by $u _{r} $ in what follows.

By the dimension formula \eqref{eqDimension}, since we need the expected
dimension of \eqref{eqModuli} to be zero, the class $A ^{/} $ satisfies:
\begin{equation*}
Maslov ^{vert} (A ^{/} ) = -2,
\end{equation*}
and we must have $\gamma ^{0} = \gamma _{0,4}.  $

\begin{notation}
 From now on, by slight abuse, $A _{0} $ refers to various section classes of various Hamiltonian structures such that the associated class $A _{0} ^{/}  $ satisfies:
   \begin{equation*}
   Maslov ^{vert} (A _{0} ^{/}  ) = -2.
   \end{equation*}
\end{notation}
\subsection {Constructing suitable $\{\mathcal{A} _{r} \} $}
\label{sec_Ar}
To get a handle on \eqref{eqModuli} we will construct
unexcited data $\mathcal{D} _{0} $ with additional
geometric properties. 

A Hamiltonian $S ^{2} $ fibration over $S ^{4} $ is classified by an element $$[g] \in \pi _{3} (\operatorname {Ham} (S ^{2}, \omega), id) \simeq \pi _{3} (PU (2),id) \simeq \mathbb{Z}.$$ Such an element determines a fibration $P _{g} $
over $S ^{4} $ via the clutching construction:
\begin{equation*}
P _{g} =   D ^{4} _{-} \times S ^{2}    \sqcup  D ^{4} _{+} \times  S ^{2} \sim,   
\end{equation*}
with $D ^{4}_- $, $D ^{4}_+ $ being 2 different names for the standard closed 4-ball $D
^{4} $, and where the equivalence relation $\sim$ is $(d, x) \sim \widetilde{g}
(d,x)$, $$ \widetilde{g}: \partial D ^{4} _{-}
\times S ^{2} \to \partial D ^{4} _{+}  \times S
^{2},  \quad   \widetilde{g} (d,x) = (d, g (d)
^{-1}  (x)).$$  We suppose that the that previously appeared
point $x _{0} = h _{\pm} (b _{0}) $, is in $D ^{4} _{+} \cap
D ^{4} _{-} \subset S ^{4} $.

From now on $P _{g} $ will denote such a
fibration for a non-trivial class $[g]$. Note that the fiber of $P _{g} $ over the base point $x _{0}  \in S ^{3} \subset D ^{4} _{\pm} $ (chosen for definition of the homotopy group $\pi _{3} (\operatorname {Ham} (S ^{2}, \omega), id) $)  has a distinguished, by the construction, identification with $S ^{2} $. 
Take $\mathcal{A}$ to be a connection on $P \simeq P _{g} $ which is trivial in the distinguished trivialization over $D ^{4} _{+}  $. 
This gives connections $$\mathcal{A}' _{r}:= (\widetilde{u}
_{r}) ^{*} \mathcal{A} $$ on $\widetilde{\mathcal{S}} _{r}$,
where $$\widetilde{u}_{r}  = \Sigma _{-} \circ u _{r}. $$

By the last axiom for the system $\mathcal{U}$ introduced in
Part I, we may choose $\{u _{r} \} $ so that the family
$\{\widetilde{u} _{r} (\mathcal{S} _{r}) \}$ induces
a singular foliation of $S ^{4} $ with the properties:
\begin{itemize}
	\item  The foliation is smooth outside $x _{0} $. Note
	that $x _{0} $ is the image by $\widetilde{u} _{r}  $ of the
	ends (images of $e _{i} $), and the image of the boundary of each $\mathcal{S} _{r} $.
	\item Each $\widetilde{u} _{r} $ is an embedding on the complement of $\widetilde{u} _{r} ^{-1} (x _{0} ) $. 
\end{itemize}
Denote by $E$ the subset $S ^{3} \subset S ^{4} $ bounding $D _{\pm} ^{4}  $.  
We may in addition suppose that each $\widetilde{u} _{r} $ intersects $E$ transversally, again on the complement of $\widetilde{u} _{r} ^{-1} (x _{0} ) $.

By the above, the preimage by $\widetilde{u}  _{r} $ of $E$  contains a smoothly embedded curve $c _{r} $ as in Figure \ref{figure:plusminus}, 
and $\widetilde{u} _{r} $ takes $c _{r} $ into $E$. 
This $c _{r} $ not uniquely determined, but we may fix
a family $r \mapsto c _{r} $, with parametrizations $$c
_{r}: \mathbb{R} \to \mathcal{S} _{r},  $$ with the
properties:
\begin{itemize}
	\item $c _{r}$ maps $(-\infty,0)  $ diffeomorphically onto
	$e _{0} (\{0\} \times (-\infty, 0))$.
	\item $c _{r}$ maps $(1,\infty)  $ diffeomorphically onto
	$e _{0}(\{1\} \times (-\infty, 0))$.
	\item  $\{c _{r} \} $ is a $C ^{0} $ continuous family in
	$r$.
\end{itemize}  
 We set:
\begin{equation*}
\widetilde{c}_{r}:= \widetilde{u} _{r} \circ c _{r}.
\end{equation*}
In Figure \ref{figure:plusminus}, the regions $R _{\pm} $ are the preimages by $\widetilde{u}  _{r} $ of $D^{4} _{\pm} \subset S ^{4} $, and $c _{r} $ bounds $R _{-} $. 
\begin{figure} [h]
 \includegraphics[width=2in]{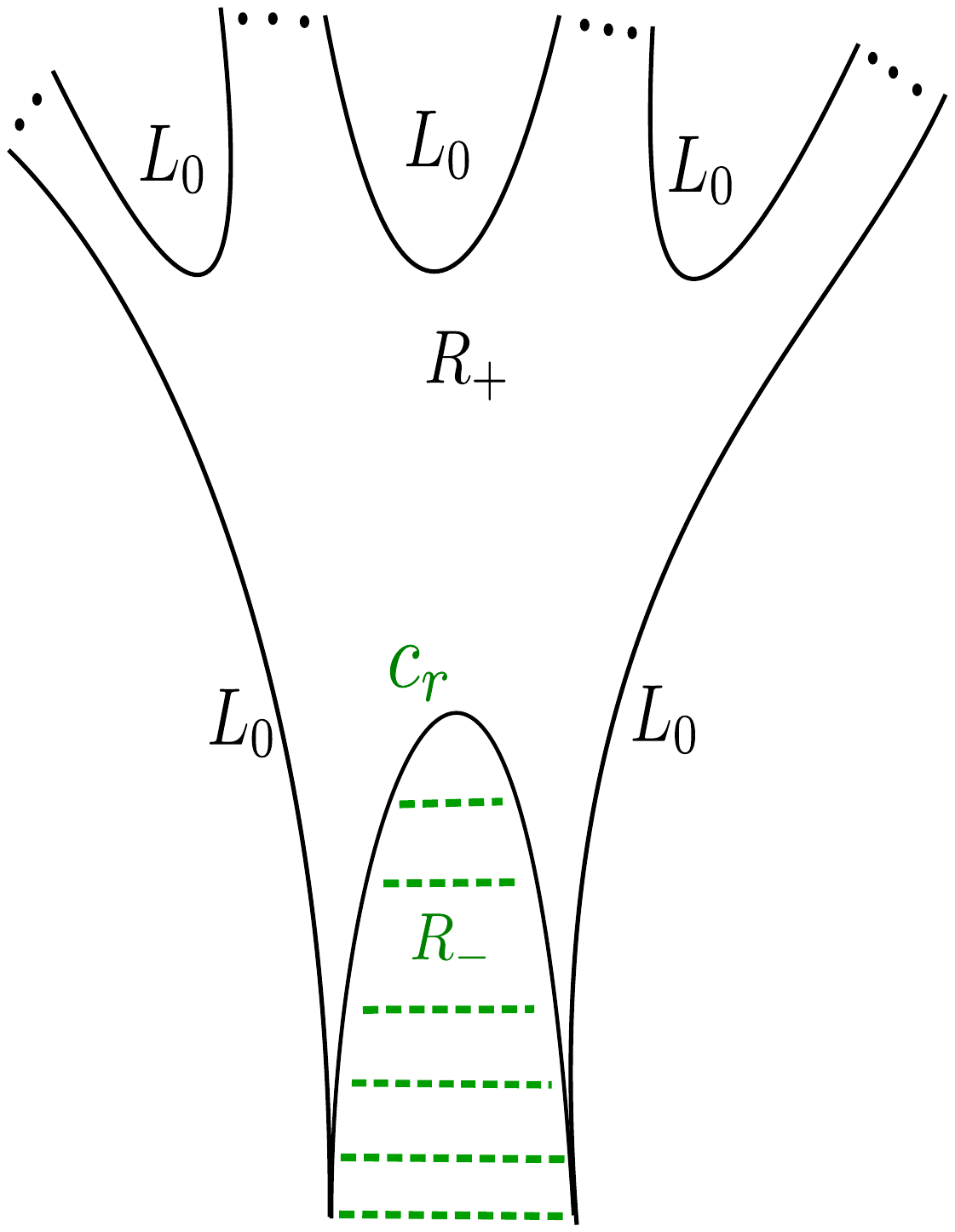}
   \caption {The labels $L _{0} $ indicate that the Lagrangian subbundle is constant with corresponding fiber $L _{0} $. The curve $c _{r} $ bounds $R _{-}$.}
\label{figure:plusminus}
\end{figure} 
It follows that  $\{\widetilde{c}  _{r} \}$ likewise induces a singular foliation of the equator $E \simeq S ^{3} $ that is smooth outside $x _{0} $.

So each $\mathcal{A}' _{r} $ is flat in the region $R _{+} $, in fact is trivial in the distinguished trivialization of $\widetilde{\mathcal{S}} _{r}  $ over $R _{+} $, corresponding to the distinguished trivialization 
of $P$ over $D ^{4} _{+}  $. Likewise, we have a distinguished trivialization of $\widetilde{\mathcal{S}} _{r}  $ over $R _{-} $,
corresponding to the distinguished trivialization of $P$ over $D^{4}_{-}$. 
In this latter trivialization let $$\phi _{r}: \mathbb{R} \to \operatorname {Ham} (S ^{2}, \omega) $$ be the holonomy path of $\mathcal{A} _{r}' $ over $c _{r} $.

Let $Lag (S ^{2})$ denote the space of 
Lagrangian equators, that is oriented Lagrangian
submanifolds Hamiltonian isotopic to the standard
oriented equator. Then by construction, $$\phi _{r}| _{(-\infty,0] \sqcup [1, \infty)}  = id, $$ 
so that we may define $$f ({r}) \in \Omega _{L _{0}}  Lag (S
^{2}))$$
by
\begin{equation} \label{eq_actionHam}
f ({r}) (t)= \phi _{r} (t) \cdot L _{0}, \quad t \in [0,1]
\end{equation}
where the right hand side means apply an element of
$\operatorname {Ham} (S ^{2}, \omega) $ to $L _{0} $ to get
a new Lagrangian. 
\begin{terminology} \label{term_generated}
We will say that $f (r)$ is
\textbf{\emph{generated by $\mathcal{A}' _{r} $}}.
\end{terminology}
 Note that
each $f (r)$ is an exact loop by construction, where exact
is the standard notion, as for example in
~\cite[Section
\ref{sec_ExactPaths}]{cite_SavelyevQuantumMaslov}.

Also by construction 
\begin{equation} \label{eq:phi}
\phi _{r} (t) =  g(\widetilde{c}  _{r}(t)),  \quad t \in [0,1],
\end{equation}
if we identify $\widetilde{c} _{r}(t) $ with an element of $S ^{3} $.

Let $D ^{2}_{0} \subset \overline{\mathcal{R}}_{4}$ be an
embedded closed disk, not intersecting the boundary
$\partial \overline{\mathcal{R}}_{4}$, so that $\partial
D ^{2} _{0}  $ is in the gluing normal neighborhood $N$ of
$\partial \overline{\mathcal{R}}_{4}  $, as defined in Part I.

So we get a continuous map $$f: D ^{2} _{0} \to \Omega _{L
_{0}} Eq (S ^{2}), $$ with $Eq (S ^{2}) \simeq S ^{2}$ denoting the space
of standard oriented equators in $S ^{2}$.
And $f(\partial D ^{2} _{0})= p_{L _{0}} $, with the right-hand side denoting the constant loop at $L _{0} $.
Then by construction, and \eqref{eq:phi} in particular, $f \simeq lag$, where $\simeq$ is a homotopy equivalence, and where
\begin{equation} \label{eq:lagdef}
lag: S ^{2} \to \Omega _{L _{0} } Lag (S ^{2})
\end{equation}
is the composition 
$$S ^{2} \xrightarrow{g'} \Omega _{id} PU (2) \to \Omega
_{L_{0} }Eq  (S ^{2}), $$ for $g'$ naturally induced by $g$,
and for the second map naturally induced by the map $$PU (2)
\to Eq (S ^{2} ), \quad \phi \mapsto \phi (L _{0}).$$
%

We then deform each $\mathcal{A}'_{r} $ to a connection
$\mathcal{A} _{r} $, which is as follows. In the region $R
_{+} $ $\mathcal{A} _{r} $ is still flat, but at each end $e
_{i} $, $\mathcal{A} _{r} $ is compatible with $\mathcal{A}
(L _{0}, L _{0})$, and so that $\mathcal{A} _{r} $ is still trivial over the boundary of $\mathcal{S} _{r} $.


Since $\widetilde{\mathcal{S}} _{r}  $ and $\mathcal{A}' _{r} $ are trivial  for $r \in \overline{\mathcal{R}} _{d} - D ^{2} _{0} $, with trivialization induced by the trivialization of $P _{+} $, and since the condition \eqref{eqAreaAd} holds, we may ensure that  
\begin{equation} \label{eq:hbarbound2}
   \energy (\mathcal{A} _{r}) < \hbar -5\kappa, 
\end{equation}
for $r$ in the complement of $D ^{2} _{0}  $. 
This give a Hamiltonian structure we denote as $$\mathcal{H} :=
\{\mathcal{H} _{r}:= (\widetilde{\mathcal{S}} _{r},
\mathcal{S} _{r}, \mathcal{L} _{r}, \mathcal{A} _{r}) \}.$$
By the paragraph after Lemma \ref{lemma:empty0}, this
extends to some unexcited data $\mathcal{D} _{0} $.
\subsection{Restructuring $\mathcal{H} $} \label{sec_restructuring}

The data $\mathcal{D}  _{0}$ is now fixed, but it remains
compute $$\mu ^{4} _{\Sigma ^{4} _{-}} (\gamma _{1}
, \ldots , \gamma _{4} ),$$ with respect to this data. To
this end we are going to restructure $\mathcal{H} $. 

For convenience we recall here some notions from
~\cite{cite_SavelyevQuantumMaslov}. 
\begin{definition}  \label{definition:concordance} 
Given a pair $\{\Theta _{r} ^{i}  \}  = \{\widetilde{{S}} _{r} ^{0} , {S}_{r} ^{0},
      \mathcal{L}_{r} ^{0}, \mathcal{A} _{r} ^{0}\}
			_{\mathcal{K}} $, of Hamiltonian structures
we say that they are \textbf{\emph{concordant}} if the
following holds. 
There is a Hamiltonian structure $$\mathcal{T} =  \{\widetilde{{T}}  _{r},
{T}_{r} , \mathcal{L}'_{r}, \mathcal{A}' _{r}
   \} _{\mathcal{K}
\times [0,1]},$$ with
an oriented diffeomorphism (in the natural sense, preserving all structure)
$$   \{\widetilde{{S}} _{r} ^{0} , {S}_{r} ^{0} ,
\mathcal{L}_{r} ^{0}, \mathcal{A} _{r} ^{0} 
\} _{ \mathcal{K} ^{op} _{0}}     \sqcup \{\widetilde{{S}} _{r} ^{1}, {S}_{r} ^{1} , \mathcal{L}_{r} ^{1}, \mathcal{A} _{r} ^{1}
\} _{ \mathcal{K} _{1}} 
\to  \{\widetilde{{T}} _{r},
{T} _{r}, \mathcal{L}'_{r}, \mathcal{A}' _{r
} \} _{\mathcal{K} \times \partial [0,1]}.
 ,$$ 
where $op$ denotes the opposite orientation.
\end{definition}

\begin{definition}
We say that a Hamiltonian structure $\{\Theta _{r} \}$ is $A$-\textbf{\emph{admissible}} if 
there are no elements $$(\sigma, r) \in \overline{\mathcal{M}} (\{\Theta _{r} \},  A),$$  for $r$ in a neighborhood of the boundary of $\mathcal{K}$. 

\end{definition}
\begin{definition}  \label{definition:isotopy} 
Given an $A$-admissible pair $\{\Theta _{r} ^{i}  \} $, $i=1,2$, of Hamiltonian structures, we say that they are $A$-\emph {\textbf{admissibly concordant}} if there
is a Hamiltonian structure $$\{\mathcal{T}
_{r}\} = \{\widetilde{{T}}  _{r},
{T}_{r} , \mathcal{L}'_{r}, \mathcal{A}' _{r}
\} _{\mathcal{K} \times [0,1]}, $$  which furnishes
a concordance, and s.t. there are no elements $(\sigma,
r) \in \overline{\mathcal{M}} (\{\Theta _{r} \},  A),$
for $r \in \partial \mathcal{K} \times [0,1] $.
\end{definition}

Applying ~\cite[Lemma \ref{lemma:gluinglowerbound}] {cite_SavelyevQuantumMaslov}, and
using \eqref{eq:hbarbound2} we get that $\mathcal{H}$ is $A _{0} $-admissible.  We now further mold this data for the purposes of computation.

First cap off the ends $e _{i} $, $i \neq 0$, of
each $\mathcal{H} _{r}$ as in the paragraph preceding
~\cite[Lemma \ref{lemma:gluinglowerbound}]{cite_SavelyevQuantumMaslov}. 
This gives a Hamiltonian structure $$\mathcal{H} ^{\wedge}:=
\{ \widetilde{S} _{r} ^{\wedge},  {S} _{r} ^{\wedge},
\mathcal{L} _{r} ^{\wedge}, \mathcal{A} _{r}^{\wedge} \}
_{\mathcal{K} = D ^{2} _{0}} ,$$ satisfying 
\begin{equation}\label{eq_areahbar2}
\energy
(\mathcal{A} _{r} ^{\wedge}  ) + \kappa < \hbar, 
\end{equation}
for each
$r$. Again by ~\cite[Lemma
\ref{lemma:gluinglowerbound}]{cite_SavelyevQuantumMaslov} $\mathcal{H} ^{\wedge}$ is $A _{0} $-admissible. 
%

Let
$$ev (\mathcal{H} ^{\wedge}, A _{0}) \in CF (L _{0}, L _{0}, \mathcal{A} (L _{0},
L _{0}))$$ be as in
~\cite[Lemma
\ref{def_EvaluationTotal}]{cite_SavelyevQuantumMaslov}.

Using the PSS maps as described in
~\cite[Section \ref{sec_Gluing Hamiltonian structures with
estimates}]{cite_SavelyevQuantumMaslov}, corresponding to the caps at the
ends (needed for the construction $\mathcal{H}  ^{\wedge}$),
and using standard gluing,  it readily follows that 
\begin{equation} \label{eq:mu=H^}
[\mu ^{4}_{\Sigma ^{4} _{-}} (\gamma _{1} , \ldots , \gamma _{4}
)] = [ev (\mathcal{H} ^{\wedge}, A _{0} )] \in FH (L _{0},
L _{0}).
\end{equation}

It remains to compute the right-hand side, to this end we
further restructure the data. 

Let $p _{1}: [0,1] \to Lag (S ^{2}) $ be the path generated
by $\mathcal{A} (L _{0}, L _{0})$, with $p _{1} $ starting
at $L _{0} $, and where generated is as in the paragraph
following \eqref{eq_actionHam}. 
Suppose we have defined $p _{i-1} $. Set $L _{i-1}:=p _{i-1} (1)  $ and define $p _{i} $ to be the path in $Lag (S ^{2})$ starting at $L _{i-1} $, generated by $\mathcal{A}(L _{0}, L _{0})    $.
Now set $p _{0} :=p _{1} \cdot \ldots \cdot {p _{d}}  $, where $\cdot$ is path concatenation in diagrammatic order. 
We may assume that $L _{0}$ is transverse to $L _{4} = p _{0} (1)$ by adjusting the connection $\mathcal{A} (L _{0}, L _{0})$ if necessary. 

We construct a concordance of $\mathcal{H} ^{\wedge}$ to another
Hamiltonian structure:  $$\mathcal{H} ^{\mathfrak
{n}}:=\{\widetilde{S} _{r} ^{\wedge}, {S}_{r}^{\wedge},
\mathcal{L}_{r} ^{\mathfrak {n}}, \mathcal{A}
_{r}^{\mathfrak {n}}\},$$  whose properties are 
illustrated in Figure \ref{diagram:deformation}. 

\begin{figure} 
 \centering 
\includegraphics[width=2in]{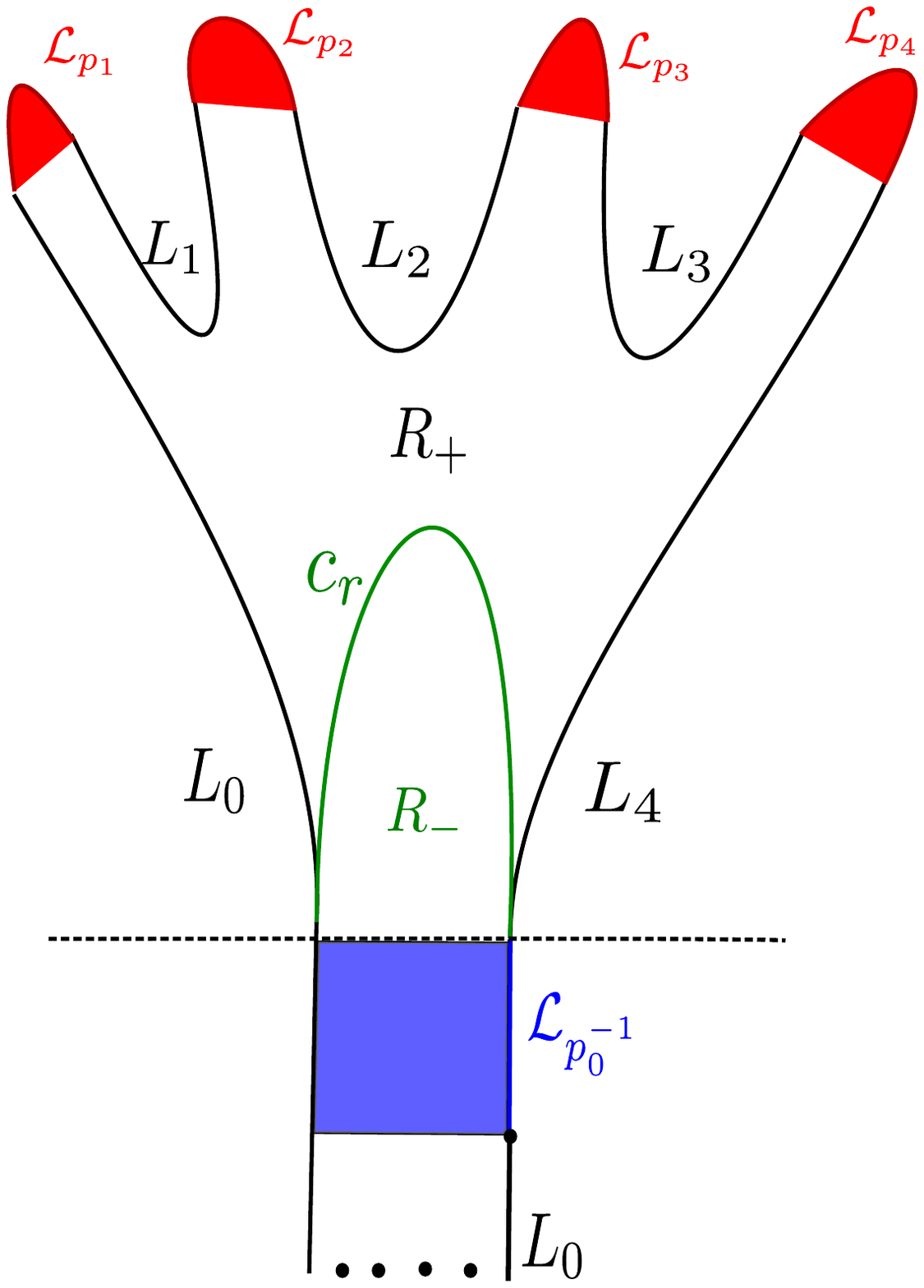}
   \caption {Over the boundary components with black labels
	 $L _{i} $ the  Lagrangian subbundle $\mathcal{L} _{r}
	 ^{\mathfrak {n}} $ is constant with corresponding fiber
	 $L _{i} $. Over the $i$'th red boundary component the
	 Lagrangian subbundle corresponds to the path of
	 Lagrangians $p _{i} $. Likewise over the
	 right boundary component of the blue region, the Lagrangian subbundle corresponds to the path of Lagrangians $p _{0} ^{-1}
	 $. In the red striped regions we have removed the
	 curvature of the connection, in the blue striped region we have added it.}
\label{diagram:deformation}
\end{figure} 
The Hamiltonian connection $\mathcal{A} _{r} ^{\mathfrak {n}}  $ satisfies the following conditions, (referring  to the Figure
\ref{diagram:deformation}):
\begin{itemize}
	\item $\mathcal{A} _{r} ^{\mathfrak {n}} $ is flat in the entire region $R _{+} $ (which includes the red shaded finger regions).
	\item The blue region is contained in the strip end chart
	at the $e _{0} $ end, which is down in the figure.
	\item Along top boundary segment of the blue region,
	contained in the dashed line,  $\mathcal{A} _{r}
	^{\mathfrak {n}}  $ is the trivial connection in the
	coordinates of corresponding end bundle chart.
	\item At the $e _{0} $ end, in the corresponding end
	bundle chart coordinates, the connection $\mathcal{A} _{r}
	^{\mathfrak {n}}  $  coincides with $\mathcal{A}
	_{r}^{\wedge}$ over $[0,1] \times (-\infty, s)$, for
	some $s<0$.
\end{itemize}
The connections $\mathcal{A} _{r} ^{\mathfrak {n}} $ are
forced by the above conditions to gain curvature in some region,
and we limit this region to the blue region of Figure
\ref{diagram:deformation}. This works more specifically as
follows.


We choose a concordance $\mathcal{T} $ from $\mathcal{H} ^{\wedge} $ to $\mathcal{H} ^{\mathfrak
{n}} $ such that for the associated family of connections
$\{\mathcal{A} _{r,t} \}$, $t \in [0,1] $, the following is
satisfied:
\begin{enumerate}
	\item $$\mathcal{A} _{r,0}= \mathcal{A} _{r} ^{\wedge},
	\quad \mathcal{A} _{r,1}= \mathcal{A} _{r} ^{\mathfrak {n}}.  $$
	\item 
	\begin{equation*} \forall r \, \forall v \in T _{z} \mathcal{S} _{r}: \frac{d}{dt} |R _{\mathcal{A} _{r,t}} (v, jv)| _{+} <0, 
\end{equation*}
for each $z \in \mathcal{S} _{r} $, except for $z$ in the
region which is blue shaded in Figure \ref{diagram:deformation}. 
	\item   
$$\forall t: |\energy (\mathcal{A} _{r,t} )-\energy (\mathcal{A}
_{r} ^{\wedge})| = 0. $$ \label{cond_boundbysumpi}
\end{enumerate}  
The last condition can be satisfied,  since the gain of
$\energy$ in the blue region is exactly equal to the loss of
$\energy$ in the red regions.

It follows by the conditions above, and the use of
~\cite[Lemma
\ref{lemma:gluinglowerbound}]{cite_SavelyevQuantumMaslov}
that $\mathcal{T} $ is an $A _{0}$-admissible concordance.

By ~\cite[Lemma \ref{lemma:gluing}]{cite_SavelyevQuantumMaslov} we have:
\begin{equation*}
   [ev (\mathcal{H} ^{\wedge}, A _{0} )] =
   [ev(\mathcal{H} ^{\mathfrak {n}}, A _{0} )].
\end{equation*}
This finishes our restructuring.
\subsection {Computing $[ev(\mathcal{H} ^{\mathfrak {n}}, A _{0} )]$}
If we stretch the neck along the  
dashed line in Figure \ref{diagram:deformation}, the upper
half of the resulting building gives us a new Hamiltonian
structure:
$$\mathcal{H} ^{0} = \{\widetilde{S} _{r} ^{0}, {S}_{r}^{0}, \mathcal{L}_{r} ^{0}, \mathcal{A} _{r}^{0}\}.$$

Let $$[ev _{\mathfrak {n}}] :=[ev(\mathcal{H} ^{\mathfrak
{n}}, A _{0} )] \in FH (L _{0}, L _{0}),$$  and let 
$$ [ev _{0} ] := [ev(\mathcal{H} ^{0}, A _{0}]  \in FH (L
_{0}, L _{4}).
$$
By standard Floer theory, and specifically the
theory of continuation maps
in Floer homology we clearly have that:
\begin{equation*}
0 \neq [ev _{\mathfrak {n}}] \iff 0 \neq [ev _{0} ].
\end{equation*}

\subsubsection{Getting a cycle of Lagrangian paths}
We may suppose that the holonomy path of $\mathcal{A} ^{0}
_{r} $ over $c _{r}$, in the distinguished trivialization over $R _{+} $,
generates $p _{0} $. Since $\mathcal{A} ^{0}$ is flat over
$R _{+}$, the latter can be insured simply by adjusting the parametrizations $\{c _{r}\}$.

Let $\mathcal{P}({L _{0}, L _{4}}) $ denote the space
of smooth exact paths in $Lag (S ^{2})$ from $L _{0} $ to $L
_{4} $. Let $$f': D _{0} ^{2} \to  \mathcal{P}({L _{0},
L _{4}}), $$ be defined as: 
\begin{equation} \label{eq:f'}
f' (r) (t)= g(\widetilde{c} _{r}(t)) \cdot p _{0}  (t),
\end{equation}
Here the right-hand side of \eqref{eq:f'} means the action
as in \eqref{eq_actionHam}: apply an element of $\operatorname {Ham} (S ^{2}, \omega) $ to a Lagrangian to get a new Lagrangian.
In this case $$f'(\partial D _{0} ^{2}) =   p_0 \in
\mathcal{P}({L _{0}, L _{4}}).$$
In particular, $f'$ represents a class 
\begin{equation}\label{eq_a}
  a \in \pi _{2} (\mathcal{P}({L _{0}, L _{4}}), p _{0}) 
\end{equation}
In what follows, we omit specifying the parameter space $D ^{2} _{0}  $ for $r$, since it will be the same everywhere.

Let $\mathcal{O}$ denote the Riemann surface with one end,
diffeomorphic to the closed disk with a single puncture on
the boundary. 
Let  $\mathcal{L} _{p}$ be the Lagrangian sub-fibration of
$\mathcal{O}  \times S ^{2}$, over the boundary induced by
$p$, and defined as follows.

as formally defined in
~\cite[Definition
\ref{defInducedLoop}]{cite_SavelyevQuantumMaslov}. F

\begin{lemma} \label{lemma:concordant} The $A _{0} $-admissible  Hamiltonian structure $\mathcal{H} ^{0} = \{\widetilde{\mathcal{S}}_{r}^{0}, \mathcal{S} _{r} ^{0},  \mathcal{L}
_{r}^{0},
\mathcal{A} ^{0}_{r}    \} $ is $A _{0} $-admissibly concordant to 
$$ \Theta' = \{
   \mathcal{O} \times S ^{2}, \mathcal{O}, \mathcal{L} _{f'
	 (r)}, \mathcal{B} _{r} 
\},$$ for certain Hamiltonian connections $\{\mathcal{B}
_{r} \}$ made explicit in the proof.
\end{lemma}

\begin{proof}
Let $R _{\pm} \subset \mathcal{S}_{r} ^{\wedge}  $ be as before. 
Fix a family of smooth deformation retractions $$ret _{r}: \mathcal{S} _{r} ^{0} \times I \to
\mathcal{S}_{r}^{0}, $$  of $\mathcal{S} _{r} ^{0} $ onto $R
^{-}  $, smooth in $r$, with $ret _{0} = id$.   We may use the smooth Riemann mapping theorem to identify
each $R ^{-} \subset \mathcal{S} _{r} ^{0}   $ with its
induced complex structure $j _{r}   $ with $(\mathcal{O},
j _{st}) $, smoothly in $r$. 

Set $ret _{r,t} = ret _{r}|_{\mathcal{S} _{r} ^{0} \times
\{t\}} $, set $\mathcal{S} _{r} ^{t} = \image ret
_{r,t}$ so that $\mathcal{S} _{r} ^{1} = R _{-}$, set $\widetilde{\mathcal{S}}_{r}^{t}
= \widetilde{\mathcal{S}}_{r}^{0}| _{\mathcal{S} _{r}
^{t}}$, i.e. the pull-back under inclusion of $\mathcal{S}
_{r} ^{t}$. Next set $\mathcal{A} _{r}^{t} = i _{r,t} ^{*}
\mathcal{A} _{r} ^{0}  $ where $i _{r,t}:
\widetilde{\mathcal{S}}_{r}^{t} \to
\widetilde{\mathcal{S}}_{r}^{0}$ is the inclusion. 

Let
$\mathcal{L} _{r} ^{t}$ be the Lagrangian sub-bundle uniquely
determined by the following conditions:
\begin{enumerate}
	\item In the end bundle chart of $\widetilde {\mathcal{S}} _{r} ^{t}$, the fiber of $\mathcal{L}
	_{r} ^{t}$ over $\{0\} \times \{t\} \subset [0,1] \times
	(-\infty, 0)$ is $L _{0} $.  The fiber of $\mathcal{L}
	_{r} ^{t}$ over $\{1\} \times \{t\} \subset [0,1] \times
	(-\infty, 0)$ is $L _{4} $.
	\item $\mathcal{L} _{r} ^{t}$ is preserved by $\mathcal{A}
	_{r} ^{t}$.
\end{enumerate}
Existence of $\mathcal{L} _{r} ^{t}$ as above, is guaranteed
by the flatness of $\mathcal{A} _{r} ^{0}$ over $R ^{+} $. 

We then get a Hamiltonian structure $$\widetilde
{\mathcal{H}}  = \{\widetilde{\mathcal{S}}_{r}^{t},
{\mathcal{S}}_{r}^{t}, \mathcal{L}
_{r}^{t}, \mathcal{A} ^{t}_{r}    \} _{r,t}.$$  
Given the Riemann mapping argument above,
$ \{{\mathcal{S}}_{r}^{1} \simeq R _{-}    \} \simeq
\{\mathcal{O} \}$ smoothly in $r$.
So $\widetilde
{\mathcal{H}}$ is a concordance 
between $\mathcal{H} ^{0} $ and $$\{
\mathcal{O} \times S ^{2}, \mathcal{O}, \mathcal{L}_{f'
(r)}, \mathcal{B} _{r}   \},$$ where $\mathcal{B} _{r}$ is
$\mathcal{A} _{r} ^{1}$ under identifications. Finally, the usual application of ~\cite[Lemma
\ref{lemma:gluinglowerbound}]{cite_SavelyevQuantumMaslov}
gives that $\widetilde {\mathcal{H} }$ is an $A _{0}$-admissible concordance.
\end{proof}

\section{Finishing up the proof of Lemma \ref{lemma:nontrivialelement}}
The existence of unexcited data $\mathcal{D} $ is proved in
Section \ref{sec:constructionSmallData}.
Given this existence, starting with \eqref{eq:mu=H^}  we
showed that $[\mu ^{4} _{\Sigma ^{4}} (\gamma _{1} , \ldots
, \gamma _{4} )]$ is non-vanishing in Floer homology iff
$$[ev (\mathcal{H} _{0}, A _{0})] \in HF (L _{0}, L _{4}),
$$ is non-vanishing. We then use Lemma
\ref{lemma:concordant} to identify $[ev (\mathcal{H} _{0},
A _{0})]$ with $[ev (\Theta', A _{0})]$.

Now by definitions $$[ev (\Theta', A _{0})] = \Psi (a),$$
where the right hand side is the quantum Maslov class as
defined in ~\cite{cite_SavelyevQuantumMaslov}, and where $a$
is a spherical class as in \eqref{eq_a}. Given that $g$ is
non-trivial the corresponding class $a$ is non-trivial and
so $\Psi (a)$ is non-trivial by 
Theorem ~\cite[Theorem \ref{thmNonZeroPsi}]{cite_SavelyevQuantumMaslov}.
This together with Lemma \ref{lemma:corellator} imply Lemma \ref{lemma:nontrivialelement}.
\qed
\section{Singular and simplicial connections and curvature bounds} \label{sec:qcurvature}
Let $\mathcal{A}$ be a $G$ connection on a principal $G$ bundle $P \to \Delta ^{n} $, and the Finsler norm $\mathfrak{n}$ on $\lie G$ be as in Section \ref{sec:non-metric} of the introduction. 
As previously discussed, a given system $\mathcal{U}$ in particular specifies maps: 
\begin{equation*}
u (m_1, \ldots, m_n, r,n): \mathcal{S} _{r} \to \Delta ^{n},
\end{equation*}
where $r \in \overline {\mathcal{R}} _{n}$, $\mathcal{S} _{r} $ is the fiber of ${\mathcal{S}}  ^{\circ} _{n}    $ over $r$, and where $(m_1, \ldots, m _{n} )$ is the composable chain of morphisms in $\Pi (\Delta ^{n} )$, $m _{i} $ being the edge morphism from the vertex $i-1$ to $i$.
%
Then define
\begin{equation} \label{eq:definitionArea}
   \energy _{\mathcal{U}} (\mathcal{A}) = \sup _{r}
   \energy _{\mathfrak n}  (u (m_1, \ldots, m_n, r,n) ^{*} \mathcal{A}),
\end{equation}
where $\energy _{\mathfrak n} $ on the right hand side is as defined in equation \eqref{eq:areasurface}.
In the case $G= \operatorname {Ham} (M,\omega)$ we take
$$\mathfrak n: \lie \operatorname {Ham} (M, \omega) \to \mathbb{R}$$ to be $$\mathfrak n (H) = |H| _{+} =  \max _{M} H. $$

Let $\omega$ be the area 1 Fubini-Study symplectic 2-form on
$M=\mathbb{CP} ^{1}$. Then the pull-back by the natural map
$$\lie h: \lie PU (2) \to \lie \operatorname {Ham} (\mathbb{CP} ^{1},\omega) \simeq C ^{\infty} _{0} (\mathbb{CP} ^{1} )$$ of the semi-norm: $|H| _{+} =  \max _{M} H$ is the operator norm on $PU (2)$, up to normalization. This will be used to get the specific form of Theorem \ref{thm:lowerboundsingular},  from the more general form here.

\subsection {Simplicial connections} \label{section:simplicialconnections}
We now introduce the notion of simplicial connections, which can partly be understood as simplicial resolutions of singular connections.
Let ${G} \hookrightarrow P \to X$ be a principal $G$ bundle, where $G$ is a Frechet Lie group.
Denote by $X_{\bullet} $ the smooth singular set of
$X$, i.e. the simplicial set whose set of $n$-simplices, $X _{\bullet} (n) $ consists smooth maps $\Sigma: \Delta ^{n} \to X$, with $\Delta ^{n} $ standard topological $n$-simplex with vertices ordered $0, \ldots, n$. 
And denote by $Simp(X _{\bullet})$ the category with objects $\cup _{n} X _{\bullet} (n)$ and with $hom (\Sigma _{0}, \Sigma _{1})$ commutative diagrams:
\begin{equation*}
\begin{tikzcd}
   \Delta ^{n} \ar [rd, "\Sigma _{0}"] \ar[r, "mor"] & \Delta ^{m} \ar[d, "\Sigma _{1}" ] \\
& X,
 \end{tikzcd}
\end{equation*}
for $mor$ a simplicial face map, that is an injective affine map preserving order of the vertices.
\begin{definition} \label{def:singualCon}
Define a \textbf{\emph{simplicial $G$-connection $\mathcal{A}$}} on $P$ to be the following data:
 \begin{itemize}
   \item For each $\Sigma: \Delta ^{n} \to X  $  in $X _{\bullet} (n) $ a smooth  $G$-connection $\mathcal{A} _{\Sigma} $ on $\Sigma ^{*}P \to \Delta ^{n} $, (a usual Ehresmann $G$-connection.)
\item For a morphism $mor: \Sigma _{0} \to \Sigma _{1}  $ in $Simp (X _{\bullet} )$, 
we ask that $mor ^{*} \mathcal{A} _{\Sigma _{1}} = \mathcal{A} _{\Sigma _{0} } $. 
\end{itemize}
\end{definition}
\begin{example} \label{example:exampleSimplicial}
   If $\mathcal{A}$ is a smooth $G$-connection on
   $P$, define a simplicial connection by
   $\mathcal{A} _{\Sigma} = \Sigma ^{*}
   \mathcal{A}  $ for every simplex $\Sigma \in X
   _{\bullet} $. We call such a simplicial
   connection \textbf{\emph{induced}}.
\end{example}
If we try to ``push forward'' a simplicial connection to get a ``classical'' connection on $P$ over $X$, then we get a kind of multi-valued singular connection. Multi-valued because each $x \in X$ may be in the image of a number of $\Sigma: \Delta ^{n} \to X $ and $\Sigma$ itself may not be injective, and singular because each $\Sigma$ is in general singular so that the naive push-forward may have blow up singularities. We will call the above the naive pushforward of a simplicial connection. 

\begin{proof} [Proof of Theorem
   \ref{thm:lowerboundsingular} and Corollary \ref{corol:example}] We will prove this by way of a stronger result.
Let $P$ be a Hamiltonian fibration $S ^{2} \hookrightarrow
P \to S ^{4}  $, and $\mathcal{A}$ a simplicial $\mathcal{G}
= \operatorname {Ham} (S ^{2},  \omega) $ connection on $P$.
Denote by $\sigma ^{1} _{0} \in S ^{4} _{\bullet} $ the degenerate $1$-simplex at $x _{0}$, in other words the constant map: $\sigma ^{1} _{0}: [0,1] \to x _{0}. $  Let $\kappa$ be the $L ^{\pm} $-length of the holonomy path of $\mathcal{A} _{\sigma ^{1} _{0}}$ over $[0,1]$. 

Finally, let $\Sigma _{\pm} \in S ^{4} _{\bullet} (4)  $ be
a complementary pair as in Section
\ref{section:model}.
The connection $\mathcal{A} $ gives us a simplicial connection  as in Example
\ref{example:exampleSimplicial}.
By an inductive procedure as
in Part I, Lemma 5.6, we may find 
perturbation data $\mathcal{D}$ for $P$ so that with respect
to $\mathcal{D} $ the following is satisfied. 
\begin{align} 
& \forall r: pr _{1} \mathcal{F} (L ^{0} _{0}, \ldots, L ^{n} _{0},  \Sigma _{\pm},r)  \simeq _{\delta}  
u (m_1, \ldots, m_s, r,n) ^{*} \mathcal{A} _{\Sigma _{\pm} }, \label{eq:F=G} \\
 & \mathcal{A} (L _{0}, L _{0}) \simeq _{\delta}  \mathcal{A} _{\sigma ^{1} _{0}}, \label{eq:F=G2}
\end{align}
where $L ^{i} _{0}  $ are the objects as before,
where $\simeq _{\delta} $ means $\delta$-close in the metrized $C ^{\infty} $ topology, and $\delta$ is as small as we like.  Here we are using notation of Part I as before. Set $$\widetilde{u} _{r}:= \Sigma _{-}  \circ u(m_1, \ldots, m _{4},r,4),$$
so $\widetilde{u} _{r}: \mathcal{S} _{r} \to S ^{4}$.
Set $\widetilde{\mathcal{S}} _{r}:= \widetilde{u} _{r} ^{*} P, $
set 
$\mathcal{A}'_{r}:= pr _{1} \mathcal{F} (L ^{0} _{0}, \ldots, L ^{n} _{0},  \Sigma _{-},r)$
and set $$\{\Theta _{r}\} :=  \{\widetilde{\mathcal{S}} _{r},
   {\mathcal{S} _{r}}, \mathcal{L} _{r}, \mathcal{A}'_{r}\}. $$
\begin{definition}\label{def:perfect}
Let $\delta $ and $\mathcal{A} (L
_{0}, L _{0})$ be as above. We say that $\mathcal{A}$ is \textbf{\emph{perfect}} if 
the following holds. For every arbitrarily small $\delta$  $\mathcal{A}
(L _{0}, L _{0}  )$ can be chosen  so that the corresponding Floer chain complex $CF (L _{0},
   L _{1}, \mathcal{A} (L _{0}, L
   _{0}) )  $ is perfect.
\end{definition}
\begin{theorem} \label{prop:alternative} 
Let $\mathcal{A}$ be a perfect simplicial Hamiltonian
connections on $P$. If 
$P$ is non-trivial as a Hamiltonian bundle
 then 
\begin{equation*}
   (\energy _{\mathcal{U}}  (\mathcal{A} _{\Sigma _{+}  })
	 \geq \hbar -5 \kappa) \lor (\energy _{\mathcal{U}}  ( \mathcal{A} _{\Sigma _{-}} ) \geq \hbar - 5\kappa),
\end{equation*}
\end{theorem}
\begin{proof}
Suppose 
\begin{equation} \label{eq:arealesshalf}
\energy _{\mathcal{U}}  (\mathcal{A} _{\Sigma _{+} } ) < \hbar - 5 \kappa.
\end{equation}
Then by \eqref{eq:F=G}, \eqref{eq:F=G} and by
~\cite[Lemma
\ref{lemma:gluinglowerbound}]{cite_SavelyevQuantumMaslov} $\mathcal{D}$, as defined
above, can be assumed to be unexcited provided $\delta$ is chosen to be sufficiently small. Take the unital replacement as in Lemma \ref{lemma:corellator}. Since we know that $K (P)$ does not admit a section by Theorem \ref{thm:noSection}, the simplex $T$ of the Lemma \ref{lemma:corellator} does not exist. 
Hence, again by this lemma 
$$ev (\{\Theta _{r} \}, A _{0}) = [\mu ^{4}
_{\Sigma _{-}} (\gamma _{1},  \ldots,  \gamma _{4} )] \neq 0.$$  
In particular $$\overline {{\mathcal{M}}}(\{\Theta _{r} \}, A _{0} ) \neq \emptyset.$$ 

So by ~\cite[Lemma
\ref{lemma:gluinglowerbound}]{cite_SavelyevQuantumMaslov} there exists an $r _{0}$ so that
\begin{equation} \label{eq:ineq1}
\energy (\mathcal{A}'_{r _{0}}) \geq \hbar - 5 \kappa'.
\end{equation}
where $\kappa'$ denotes the $L ^{\pm} $ length of the
holonomy path in $\operatorname {Ham} (S ^{2},  \omega) $ of $\mathcal{A} (L _{0}, L _{0}) $. 
By \eqref{eq:F=G2} $\kappa' \to \kappa$ as $\delta
\to 0$.
By \eqref{eq:F=G}, \eqref{eq:ineq1}
passing to the limit as $\delta \to 0$ we get:
$$\energy _{\mathcal{U}}  ( \mathcal{A} _{\Sigma _{-}} ) \geq \hbar - 5\kappa.$$
\end{proof}
\begin{corollary}
   \label{corol:alternative} 
Let $\mathcal{A}$ be a $PU (2) $ connection 
on a non-trivial principal $PU (2) $
bundle $P \to S ^{4}$. 
 Then 
\begin{equation*}
   (\energy _{\mathcal{U}}  (\mathcal{A} _{\Sigma
   _{+}  })  \geq \hbar -5 \kappa) \lor (\energy
   _{\mathcal{U}}  ( \mathcal{A} _{\Sigma _{-}} )
   \geq \hbar - 5\kappa).
\end{equation*}
\end{corollary}
\begin{proof}
 A simplicial $PU (2)$ connection $\mathcal{A}$ on a
principal $PU (2)$ bundle $PU (2) \hookrightarrow
P' \to S ^{4} $ is automatically perfect, when
understood as a Hamiltonian connection on the
associated bundle $S ^{2}  \hookrightarrow P \to S
^{4} $. So that this is an immediate consequence
of the theorem above.
\end{proof}
To prove Corollary \ref{corol:example},
we just note that
for an induced simplicial Hamiltonian connection
$\mathcal{A} $, as defined in
Example \ref{example:exampleSimplicial}, $\mathcal{A} _{\sigma ^{1} _{0}}$
is trivial.  And hence $\mathcal{A}$ is
automatically perfect.  So that this corollary follows by Theorem \ref{prop:alternative}.

\end{proof}

\appendix 

\section {Homotopy groups of Kan complexes} \label{appendix:Kan}
For convenience let us quickly review Kan
complexes just to set notation.  This notation is
also used in Part I. Let $$\Delta ^{n}_{\bullet } (k):= hom _{\Delta} (k, n),  $$  be the standard representable $n$-simplex,
where $\Delta$ is as in Section \ref{sec_preliminaries}. 

Let $\Lambda ^{n} _{k} \subset \Delta ^{n} _{\bullet } $ denote the sub-simplicial set corresponding to the ``boundary'' of $\Delta ^{n} _{\bullet}   $ with the $k$'th face removed, $0 \leq k \leq n$. 
By $k'th$ face we mean the face opposite
to the $k$'th vertex. 
Let $X _{\bullet}$ be an abstract simplicial set.
A simplicial map $$h: \Lambda ^{n} _{k} \subset \Delta ^{n} _{\bullet}  \to X _{\bullet}  $$ will be called a \textbf{\emph{horn}}.
A simplicial set $S _{\bullet}$ is said to be a \textbf{\emph{Kan complex}} if for all $n,k \in \mathbb{N}$ given a diagram  with solid arrows
\begin{equation*} 
\begin {tikzcd}
   \Lambda ^{n} _{k} \ar [r, "h"]  \ar [d, "i"] & S _{\bullet} \\
 \Delta ^{n} _{\bullet} \ar [ur, dotted, "\widetilde{h}"]  &, \\ 
\end{tikzcd}
\end{equation*}
there is a dotted arrow making the diagram commute. 
The map $\widetilde{h} $ will be called \textbf{\emph{the Kan filling}} of the horn $h$. The $k$'th face of $\widetilde{h} $ will be called \textbf{\emph{Kan filled face along $h$}}. 
As before we will denote Kan complexes and $\infty$-categories by calligraphic letters.

Given a pointed Kan complex $(\mathcal{X} ,x) $ and $n \geq 1$ the
\emph{$n$'th simplicial homotopy group} of
$(\mathcal{X},x) $: $\pi _{n} (\mathcal{X}, x ) $ is defined to be the set
of equivalence classes of maps
$$\Sigma: \Delta ^{n} _{\bullet}   \to \mathcal{X},$$ 
such that $\Sigma$ takes $\partial \Delta
^{n}_{\bullet} $ to $x _{\bullet} $, with the
latter denoting the image of $\Delta ^{0}
_{\bullet} \to \mathcal{X}   $, induced by the vertex inclusion $x \to X$.

More precisely, we have a commutative diagram:
\begin{equation*} 
\begin{tikzcd}
   \Delta ^{n} _{\bullet}  \ar [rd, "\Sigma"] \ar[r] & \Delta ^{0} _{\bullet}  \ar[d, "x" ] \\
& \mathcal{X}.
 \end{tikzcd}
\end{equation*}
\begin{example}
   \label{example:singularset} 
When $\mathcal{X} = X _{\bullet} $ is the
singular simplicial set of a  
topological space $X$, the maps above are in complete correspondence with maps:
\begin{equation*}
\Sigma: \Delta ^{n}  \to X,   
\end{equation*}
taking the topological boundary of $\Delta ^{n} $ to $x$.

\end{example}

For $X _{\bullet }$ general simplicial set, a pair of maps $\Sigma _{1}: \Delta ^{n}
_{\bullet}   \to X _{\bullet}, \Sigma _{2}: \Delta
^{n} _{\bullet}   \to X _{\bullet},$ are
equivalent if there is a diagram, called simplicial homotopy: 
\begin{equation*}
\begin{tikzcd}
 \Delta ^{n} _{\bullet} \ar[rd, "\Sigma _{1}"] \ar
 [d, "i_0"] & \\
 \Delta ^{n} _{\bullet}  \times I _{\bullet}   \ar
 [r,""] & X _{\bullet} \\
 \Delta ^{n} _{\bullet}  \ar [u, "{i_1}"]  \ar
 [ru, "{\Sigma _{2}}" ].  &  
\end{tikzcd}
\end{equation*}
such that $\partial \Delta ^{n}_{\bullet} \times I _{\bullet}$ is taken by $H$
to $x _{\bullet}  $.
The simplicial homotopy groups of a Kan complex $
(\mathcal{X},x)$ coincide with the
classical homotopy groups of the geometric
realization $(|\mathcal{X}|,x)$. 
\begin{proof} [Proof of Lemma \ref{lemma:kanfib}]  
We refer the reader to Part I, Appendix A.2,
for more details on the notions here.  We prove a stronger claim. 
\begin{lemma}
   \label{lemma:innerfibkan} Let $p: \mathcal{Y}  \to \mathcal{X}$
be an inner fibration of quasi-categories
   $\mathcal{Y},\mathcal{X}$,
with $\mathcal{X}$ a Kan complex. And let 
     $K (\mathcal{Y}) \subset \mathcal{Y}
 $ denote  the maximal Kan subcomplex.
Then $p: K (\mathcal{Y})  \to \mathcal{X} $ is a Kan
fibration.
\end{lemma}
The above is probably well known,  but it is simple
to just provide the proof for convenience.   
\begin{proof}
By definition of an inner fibration,  whenever we are given a commutative diagram with
solid arrows and with $0<k<n$:
\begin{equation} 
\begin{tikzcd} \label{diagram:inner2}
   \Lambda ^{n} _{k}  \ar [r, "\sigma"]  \ar[d, hookrightarrow] &
K (\mathcal{Y}) \ar [r, hookrightarrow] &  \mathcal{Y} 
\ar [dl, "p"]  \\
 \Delta ^{n} _{} \ar [r] \ar [urr, "\Sigma ", dashrightarrow] & \mathcal{X}
 _{}, \\
\end{tikzcd} 
\end{equation}
there exists a dashed arrow $\Sigma $ as
indicated, making the whole diagram commutative.  
When $n>2$ the edges, i.e. 1-faces, of $\Sigma$
   are all automatically isomorphisms in $\mathcal{Y} _{}$, as $\Sigma$ extends
$\sigma$, 
and all edges of $\sigma$ are isomorphisms
by definition.  For $n=2$  the edges of $\Sigma$
are either edges of 
$\sigma$, or are compositions of edges of $\sigma$
in the quasi-category $\mathcal{Y}$, and hence again
always invertible. 

It follows that 
$\Sigma$ maps into $K (\mathcal{Y}) \subset \mathcal{Y} $. Since the starting
diagram was arbitrary, we just proved
that $p: K (\mathcal{Y} _{}) \to \mathcal{X} _{} $ is an inner fibration.  In particular the pre-images $p
^{-1} (\Sigma (\Delta  ^{n} _{})) \subset K
(\mathcal{Y} _{ }) 
 $ are quasi-categories, for all $n$, where $\Sigma: \Delta  ^{n}
_{} \to \mathcal{X}$ is any $n$-simplex,  see
Part I, Appendix A.2. But $K (\mathcal{Y}) $ is a Kan
complex, so that also the above pre-images $p
^{-1} (\Sigma (\Delta  ^{n} _{})) $ are Kan
complexes. It readily follows from this that   
$p: K (\mathcal{Y} _{ }) \to \mathcal{X} _{ } $ is a Kan
fibration.
\end {proof}  
The main lemma then follows, since if $p: \mathcal{Y} _{} \to \mathcal{X} _{ } $ is a
Cartesian fibration,  it is in particular an
inner fibration.    
\end{proof}
\bibliographystyle{siam}     
\bibliography{link.bib}

\def\cprime{$'$}
\begin{thebibliography}{10}

\bibitem{cite_AtiyahBottTheYang-MillsequationsoverRiemannsurfaces}
{\sc M.~F. Atiyah and R.~Bott}, {\em {The Yang-Mills equations over Riemann
  surfaces.}}, Philos. Trans. R. Soc. Lond., A, 308 (1983), pp.~523--615.

\bibitem{cite_ReeseHarverSingular}
{\sc R.~{Harvey} and H.~B. jun. {Lawson}}, {\em {A theory of characteristic
  currents associated with a singular connection}}, {Bull. Am. Math. Soc., New
  Ser.}, 31 (1994), pp.~54--63.

\bibitem{cite_LurieHighertopostheory}
{\sc J.~Lurie}, {\em {Higher topos theory}}, {Annals of Mathematics Studies
  170. Princeton, NJ: Princeton University Press. }, 2009.

\bibitem{cite_FilteredHomotopyTypeFaria}
{\sc J.~F. Martins}, {\em On the homotopy type and the fundamental crossed
  complex of the skeletal filtration of a {CW}-complex}, Homology Homotopy
  Appl., 9 (2007), pp.~295--329.

\bibitem{cite_YangMillsBlackHoles}
{\sc F.~Naeimipour, B.~Mirza, and F.~Jahromi}, {\em Yang–mills black holes in
  quasitopological gravity}, The European Physical Journal C, 81 (2021).

\bibitem{cite_RiehlAmodelstructureforquasi-categories}
{\sc E.~Riehl}, {\em A model structure for quasi-categories},
  \url{https://emilyriehl.github.io/files/topic.pdf}.

\bibitem{cite_RognesK4Zistrivial}
{\sc J.~Rognes}, {\em {{\(K_4(\mathbb{Z})\)}} is the trivial group}, Topology,
  39 (2000), pp.~267--281.

\bibitem{cite_SavelyevAlgKtheory}
{\sc Y.~Savelyev}, {\em {Hamiltonian elements in algebraic $K$-theory}},
  \url{http://yashamon.github.io/web2/papers/KtheoryCycles.pdf}.

\bibitem{cite_SavelyevQuantumMaslov}
{\sc Y.~Savelyev}, {\em {Quantum Maslov classes}},
  \url{http://yashamon.github.io/web2/papers/quantumMaslov.pdf}.

\bibitem{cite_SavelyevGlobalFukayaCategoryI}
\leavevmode\vrule height 2pt depth -1.6pt width 23pt, {\em Global {Fukaya}
  category {I}}, Int. Math. Res. Not., 2023 (2023), pp.~18302--18386.

\bibitem{cite_SingularSibnerSibner}
{\sc L.~M. {Sibner} and R.~J. {Sibner}}, {\em {Classification of singular
  Sobolev connections by their holonomy}}, {Commun. Math. Phys.}, 144 (1992),
  pp.~337--350.

\bibitem{cite_ToenThehomotopytheoryofdg-categoriesandderivedMoritatheory}
{\sc B.~To\"en}, {\em {The homotopy theory of dg-categories and derived Morita
  theory.}}, Invent. Math., 167 (2007), pp.~615--667.

\end{thebibliography}
\end{document}